\documentclass[a4paper,reqno]{amsart}

\usepackage[english]{babel}
\usepackage[all]{xy}
\usepackage{amssymb}
\usepackage{verbatim}
\usepackage{geometry}
\usepackage{enumerate}
\usepackage{hyperref}

\theoremstyle{plain}
\newtheorem*{teo1}{Theorem 1}
\newtheorem*{teo2}{Theorem 2}
\newtheorem{lem}{Lemma}[section]
\newtheorem{teo}[lem]{Theorem}
\newtheorem{prop}[lem]{Proposition}
\newtheorem{cor}[lem]{Corollary}

\theoremstyle{definition}
\newtheorem*{dfn*}{Definition}
\newtheorem{dfn}[lem]{Definition}

\theoremstyle{remark}
\newtheorem{oss}[lem]{Remark}
\newtheorem{ex}[lem]{Example}

\let\oldmarginpar\marginpar
\renewcommand\marginpar[1]{\-\oldmarginpar[\raggedleft\footnotesize
    #1]{\raggedright\footnotesize #1}}

\DeclareMathOperator{\End}   {End}
\DeclareMathOperator{\diag}  {diag}
\DeclareMathOperator{\Id}    {Id}
\DeclareMathOperator{\Pic}   {Pic}
\DeclareMathOperator{\Supp}  {Supp}

\newcommand{\Quot}[2] {\ensuremath{\raisebox{.75ex}{\ensuremath{#1}} \! \Big / \! \raisebox{-.75ex}{\ensuremath{#2}}}}

\newcommand{\PQ}{\Pi_G}

\newcommand{\LB}[1]{\mathrm{Lb}(#1,G_\mathrm{ad})}
\newcommand{\LBb}[1]{\overline{\mathrm{Lb}}(#1,G_\mathrm{ad})}
\newcommand{\LBQ}[1]{\mathrm{Lb}(#1,G)}
\newcommand{\LBQb}[1]{\overline{\mathrm{Lb}}(#1,G)}

\newcommand{\PQm}[1] {\PQ^{\mathrm{m}}(#1)}
\newcommand{\PQo}[1] {\PQ^+(#1)^\circ}

\begin{document}

\title[Normality and smoothness of simple linear group compactifications]{Normality and smoothness\\of simple linear group compactifications}

\author{Jacopo Gandini and Alessandro Ruzzi}

\date{\today}

\email{gandini@mat.uniroma1.it}
\curraddr{\textsc{Dipartimento di Matematica ``Guido Castelnuovo''\\
                    ``Sapienza'' Universit\`a di Roma\\
                    Piazzale Aldo Moro 5\\
                    00185 Roma, Italy}}

\email{Alessandro.Ruzzi@math.univ-bpclermont.fr}
\curraddr{\textsc{Universit\'e Blaise Pascal - Laboratoire de Math\'ematiques\\ UMR 6620 - CNRS
Campus des C\'ezeaux \\ 63171 Aubi\`ere cedex\\ France}}

\begin{abstract}
Given a semisimple algebraic group $G$, we characterize the normality and the smoothness of its simple linear compactifications, namely those equivariant
$G\times G$-compacti\-fications possessing a unique closed orbit which arise in a projective space of the shape $\mathbb{P}(\End(V))$, where $V$ is a finite
dimensional rational $G$-module. Both the characterizations are purely combinatorial and are expressed in terms of the highest weights of $V$. In
particular, we show that $\mathrm{Sp}(2r)$ (with $r \geqslant 1$) is the unique non-adjoint simple group which admits a simple smooth compactification.
\end{abstract}

\maketitle

\section*{Introduction.}

Let $G$ be a connected semisimple algebraic group defined
over an algebraically closed field $\Bbbk$ of characteristic zero.
Fix a maximal torus $T\subset G$ and a Borel subgroup $B\supset T$,
denote $\mathcal{X}(T)$ the character lattice of $T$
and denote $\mathcal{X}(T)^+$ the semigroup of dominant characters.
If $\Pi \subset \mathcal{X}(T)^+$ is a finite set of dominant weights and if $V(\mu)$
denotes the simple $G$-module of highest weight $\mu$, set
$V_\Pi = \bigoplus_{\mu \in \Pi} V(\mu)$ and consider the $G\times G$-variety
$$
	X_\Pi = \overline{(G \times G) [\Id]} \subset \mathbb{P}(\End(V_\Pi)),
$$
which is a compactification of a quotient of $G$.
Suppose moreover that $X_\Pi$ is a simple $G\times G$-variety,
i.e. that it possesses a unique closed orbit:
the aim of this paper is to characterize the normality and the smoothness
of $X_\Pi$ by giving explicit combinatorial conditions on the set of weights $\Pi$.

In \cite{Ti}, Timashev studied the general situation of a connected reductive group
without any assumption on the number of closed orbits,
giving necessary and sufficient conditions for the normality and for the smoothness of these
compactifications, however the conditions of normality in particular are not completely
explicit. On the other hand, in \cite{BGMR} the authors
together with P.~Bravi and A.~Maffei studied the case of
a simple compactification of a connected semisimple adjoint group $G$:
in that case the conditions of normality and smoothness
were considerably simplified and this has been
the starting point of the paper.

To explain our results we need some further notation.
Let $\Phi$ be the root system associated to $T$
and let $\Delta\subset \Phi$ be the set of simple roots associated to $B$,
which we identify with the set of vertices of the Dynkin diagram of $\Phi$.
Recall the \textit{dominance order} and the \textit{rational dominance order} on $\mathcal{X}(T)$,
respectively defined by $\nu \leqslant \mu$ if and only if $\mu - \nu \in \mathbb{N}[\Delta]$ and by
$\nu \leqslant_\mathbb{Q} \mu$ if and only if $\mu - \nu \in \mathbb{Q}^+[\Delta]$.

If $\lambda \in \mathcal{X}(T)^+$, denote $\Pi(\lambda) \subset \mathcal{X}(T)$ the set of weights occurring in $V(\lambda)$ and denote $\mathcal{P}(\lambda)$ the convex hull of
$\Pi(\lambda)$ in $\mathcal{X}(T) \otimes \mathbb{Q}$. Denote
\begin{align*}
	\Pi^+(\lambda) &= \Pi(\lambda) \cap \mathcal{X}(T)^+ =
	\{ \mu \in \mathcal{X}(T)^+ \, : \, \mu \leqslant \lambda \},\\
	\PQ^+(\lambda) &= \mathcal{P}(\lambda) \cap \mathcal{X}(T)^+ =
	\{ \mu \in \mathcal{X}(T)^+ \, : \, \mu \leqslant_\mathbb{Q} \lambda \}.
\end{align*}
Define moreover the \textit{support} of $\lambda$ as the set
$\Supp(\lambda)=\{\alpha\in \Delta \, : \, \langle \lambda, \alpha^\vee \rangle \neq 0 \}.$

If $\Pi \subset \mathcal{X}(T)^+$, then $X_\Pi$ is a simple compactification of a quotient of $G$ if and only if $\Pi$ possesses a unique maximal element w.r.t.
$\leqslant_\mathbb{Q}$, whereas $X_\Pi$ is a simple compactification of a quotient of the adjoint group $G_\mathrm{ad}$ if and only if $\Pi$ possesses a unique maximal
element w.r.t. $\leqslant$. We will say that a subset $\Pi \subset \mathcal{X}(T)^+$ is \textit{simple} if it possesses a unique maximal element w.r.t. $\leqslant_\mathbb{Q}$,
and we say that $\Pi$ is \textit{adjoint} if it possesses a unique maximal element w.r.t. $\leqslant$. For instance, the subsets $\Pi^+(\lambda)$ and $\PQ^+(\lambda)$
are both simple and $\Pi^+(\lambda)$ is also adjoint.

Suppose that $\Pi$ is a simple adjoint subset with maximal element $\lambda$. The normality of $X_\Pi$ was characterized in \cite{BGMR} by introducing a
subset $\LB{\lambda} \subset \Pi^+(\lambda)$, called the set of \textit{adjoint little brothers} of $\lambda$, with the following property: $X_\Pi$ is a normal
compactification of $G_\mathrm{ad}$ if and only if $\Pi \supset \LB{\lambda}$. Every adjoint little brother of $\lambda$ is a weight covered by $\lambda$ (i.e. it
is maximal in $\Pi^+(\lambda) \smallsetminus \{\lambda\}$ w.r.t. $\leqslant$) which arises correspondingly to a non-simply-laced connected component $\Delta' \subset \Delta$ and
is defined in a purely combinatorial way by $\Supp(\lambda) \cap \Delta'$ (see Definition~\ref{def: twin}).

In order to explain how previous characterization of the normality extends to the non-adjoint case, we need to introduce a partial order on $\mathcal{X}(T)$
which is less fine than the dominance order. If $\beta = \sum_{\alpha \in \Delta} n_\alpha \alpha \in \mathbb{Q}[\Delta]$, define its \textit{support over} $\Delta$ as
$\Supp_\Delta(\beta) = \{ \alpha \in \Delta \, : \, n_\alpha \neq 0 \}$. If $\lambda \in \mathcal{X}(T)^+$, consider the set of positive roots which are non-orthogonal to
$\lambda$
$$
	\Phi^+(\lambda) = \{ \beta \in \Phi^+ \, : \, \Supp_\Delta(\beta) \cap \Supp(\lambda) \neq \varnothing \}
$$
and consider the associated partial orders on $\mathcal{X}(T)$:
\begin{align*}
	\nu \leqslant^\lambda \mu \qquad &\text{ if and only if } \qquad \mu - \nu \in \mathbb{N}[\Phi^+(\lambda)],\\
	\nu \leqslant^\lambda_\mathbb{Q} \mu \qquad &\text{ if and only if } \qquad \mu - \nu \in \mathbb{Q}^+[\Phi^+(\lambda)].
\end{align*}
These orderings arise naturally in the representation theory of $G$: for instance if $\mu \in \Pi(\lambda)$ then $\mu \leqslant^\lambda \lambda$, while if $\mu \in
\mathcal{P}(\lambda)$ then $\mu \leqslant^\lambda_\mathbb{Q} \lambda$ (see Proposition \ref{prop: lambda dominanza in natura}). Define the set of \textit{little brothers} of $\lambda$
w.r.t. $G$ as follows:
$$
	\LBQ{\lambda} = \left\{\mu \in \big(\PQ^+(\lambda) \smallsetminus \Pi^+(\lambda) \big) \cup \LB{\lambda} \, : \, \mu \text{ is maximal w.r.t. } \leqslant_\mathbb{Q}^\lambda \right\}.
$$
Then the normality of the variety $X_\Pi$ is characterized as follows, where for $\alpha \in \Delta$ we denote by $r_\alpha$ the number of vertices of the
connected component of the Dynkin diagram of $G$ containing $\alpha$ and where for $I \subset \Delta$ we denote by $I^\circ = \{\alpha \in I \, : \, \langle \alpha,
\beta^\vee \rangle = 0 \;\; \forall \beta \in \Delta \smallsetminus I\}$ the \textit{inner part} of $I$.

\begin{teo1} [see Theorem~\ref{teo: caratterizzazione normalita}]
Let $\Pi\subset \mathcal{X}(T)^+$ be a simple subset with maximal element $\lambda$ and suppose that $\langle \lambda , \alpha^\vee \rangle \geqslant r_\alpha$ for every $\alpha
\in \Supp(\lambda) \smallsetminus \Supp(\lambda)^\circ$. Then the variety $X_\Pi$ is a normal compactification of $G$ if and only if $\Pi \supset \LBQ{\lambda}$.
\end{teo1}

The assumption on the coefficients of $\lambda$ in previous theorem involves no loss of generality. If indeed we define $\Pi' = \{\mu + \lambda \, : \, \mu \in
\Pi\}$, then the varieties $X_\Pi$ and $X_{\Pi'}$ are equivariantly isomorphic: rather than $\Pi$, the variety $X_\Pi$ depends only on the set of simple
roots $\Supp(\lambda)$ and on the set of differences $\{\mu - \lambda \, : \, \mu \in \Pi\}$ (see Proposition \ref{prop: criterio tensoriale morfismi}).

Roughly speaking, if $\langle \lambda , \alpha^\vee \rangle \geqslant r_\alpha$ for every $\alpha \in \Supp(\lambda) \smallsetminus \Supp(\lambda)^\circ$, then the set of differences
$\{\mu - \lambda \, : \, \mu \in \LBQ{\lambda}\}$ is constant and depends only on $\Supp(\lambda)$, while otherwise it may happen that $\PQ^+(\lambda)$ ``misses'' some
maximal element of the partial order $\leqslant^\lambda_\mathbb{Q}$. This fact is linked to the problem of the surjectivity of the multiplication map between sections of
globally generated line bundles on the canonical compactification of $G$: while the multiplication is always surjective in the case of an adjoint group
(this was shown by S.~S.~Kannan in \cite{Ka}), the surjectivity of the multiplication may fail in the general case (see Example \ref{ex: controesempio suriettivita} and Proposition \ref{prop: generatori projcoord}).

In the case of an odd orthogonal group, a combinatorial classification of its simple linear compactifications  was given by the first author in \cite{Ga}
by means of a partial order tightly related to $\leqslant^\lambda$, which makes use only of the positive long roots which are non-orthogonal to the dominant weight
$\lambda$. A similar classification should be expectable in the case of any (non simply-laced) semisimple group by using similar partial orders.

While adjoint groups possess many simple smooth compactifications, $\mathrm{Sp}(2r)$ is essentially the unique non-adjoint group possessing such
compactifications. This is the content of the following theorem: together with the characterization of the smoothness given in \cite{BGMR} in the adjoint
case (see Theorem \ref{teo: smoothness caso aggiunto}), it gives a combinatorial characterization of the smoothness of a simple linear group
compactification in the general case.

\begin{teo2} [see Proposition \ref{prop: LB Sp}, Lemma \ref{lem: riduzione G semplice} and Theorem \ref{teo: smoothness caso non aggiunto}]
Let $\Pi \subset \mathcal{X}(T)^+$ be simple with maximal element $\lambda$. If $G_i \subset G$ is a simple normal subgroup and if $T_i = T \cap G_i$, set $\Pi_i =
\{\mu\bigr|_{T_i} \, : \, \mu \in \Pi\}$.
\begin{itemize}
	\item[i)] If $X_\Pi$ is smooth, then $G = G_1 \times \ldots \times G_n$ is a direct product of simple groups and $X_\Pi = X_{\Pi_1} \times \ldots
\times X_{\Pi_n}$, where $X_{\Pi_i}$ is a smooth compactification of $G_i$. Moreover, every $G_i$ is either adjoint or isomorphic to $\mathrm{Sp}(2r_i)$
for some $r_i \geqslant 1$.
	\item[ii)] Suppose that $G = \mathrm{Sp}(2r)$ with $r\geq 1$. Then $X_\Pi$ is a smooth compactification of $G$ if and only if $\Supp(\lambda)$ is
connected, $\alpha_r \in \Supp(\lambda)$ and $\lambda + \omega_{r-1} - \omega_r \in \Pi$ (where we set $\omega_0 = 0$).
\end{itemize}
\end{teo2}

The paper is organized as follows. In Section 1 we introduce the variety $X_\Pi$ and we give some preliminary results: almost all of these results come
from \cite{BGMR} and \cite{Ga} and although some of them are claimed in a more general form than the original ones, the proofs are substantially the same.
In Section 2 we introduce the partial orders $\leqslant^\lambda$ and $\leqslant^\lambda_\mathbb{Q}$. After showing how these orderings arise naturally in the representation
theory of $G$, we show some properties of the elements in $\PQ^+(\lambda)$ which are maximal w.r.t. $\leqslant^\lambda$: these results will be fundamental in the
proofs of the main results of the paper. Finally, in Section 3 we characterize the normality of the variety $X_\Pi$, while in Section 4 we characterize its
smoothness.\\

\textit{Acknowledgements.} We would like to thank A.~Maffei for fruitful conversations on the subject. As well, we would like to thank the referee for his careful reading and for useful suggestions and remarks.

\section{Preliminaries.}

Let $G$ be a connected semisimple algebraic group with Lie algebra $\mathfrak{g}$ and with adjoint group $G_\mathrm{ad}$. Let $B\subset G$ be a Borel subroup and $T\subset
B$ a maximal torus, $B^-$ denotes the opposite Borel subgroup of $B$ w.r.t. $T$. Denote by $\Lambda$ the weight lattice of $G$ and $\Lambda^+\subset \Lambda$ the
semigroup of dominant weights. Denote $\Phi$ the root system associated to $T$ and $\Delta \subset \Phi$ the basis associated to $B$, denote $W$ the Weyl
group of $\Phi$. Denote $\mathcal{X}(T)$ the character lattice of $T$ and set $\Lambda_{\mathrm{rad}} = \mathbb{Z}[\Delta]$ the root lattice, set $\mathcal{X}(T)^+ = \mathcal{X}(T) \cap
\Lambda^+$ and $\mathcal{X}(T)_\mathbb{Q} = \mathcal{X}(T) \otimes \mathbb{Q}$. Denote $\Phi^\vee$ the set of coroots  and $\Delta^\vee \subset \Phi^\vee$ the basis of $\Phi^\vee$
associated to $B$. Denote $\mathcal{X}(T)^\vee$ the cocharacter lattice of $T$, denote $\Lambda^\vee$ the coweight lattice of $G$ and $\Lambda_{\mathrm{rad}}^\vee =
\mathbb{Z}[\Delta^\vee]$ the coroot lattice of $G$. If $V$ is a $G$-module and if $H \subset G$ is a closed subgroup, we denote by $V^{(H)} \subset V$ the subset of
$H$-eigenvectors.

We consider the following partial orders on $\Lambda$:
the \textit{dominance order}, defined by $\nu \leqslant \mu$
if and only if $\mu - \nu \in \mathbb{N}[\Delta]$,
and the \textit{rational dominance order}, defined by $\nu \leqslant_\mathbb{Q} \mu$
if and only if $\mu - \nu \in \mathbb{Q}^+[\Delta]$.
Notice that, if $\nu, \mu \in \Lambda^+$, then $\nu \leqslant_\mathbb{Q} \mu$
if and only if $\det(C) \nu \leqslant \det(C) \mu$,
where $C$ is the Cartan matrix of $\Phi$: indeed $\det(C) \Lambda \subset \mathbb{Z}[\Delta]$,
hence  $\mu - \nu \in \mathbb{Q}^+[\Delta]$ if and only if $\det(C)(\mu-\nu) \in \mathbb{N}[\Delta]$.

If $I \subset \Delta$, define its \emph{border} $\partial{I}$ and its
\emph{interior} $I^\circ$ as follows:
\[ \partial{I} =
\{ \alpha \in \Delta \smallsetminus I \, : \, \exists\, \beta \in I \mathrm{ \; such \,  that \; }
\langle \beta,\alpha^\vee \rangle \neq 0\},\]
\[ I^\circ = I \smallsetminus \partial{(\Delta \smallsetminus I)}.\]
We say that $\alpha \in I$ is an \textit{extremal root} of $I$ if
it is connected to at most one other element of $I$
in the Dynkin diagram of $G$.

If $\lambda \in \Lambda$, define its \emph{support} as follows
$$
	\Supp(\lambda) = \{\alpha \in \Delta \, : \, \langle \lambda, \alpha^\vee \rangle \neq 0 \}.
$$
If $\beta = \sum_{\alpha \in \Delta} n_\alpha \alpha \in \mathbb{Q}[\Delta]$,
define its \textit{support over} $\Delta$ as follows:
$$
	\Supp_\Delta(\beta) = \{ \alpha \in \Delta \, : \, n_\alpha \neq 0 \}.
$$

If $\lambda \in \mathcal{X}(T)^+$, denote $V(\lambda)$ the simple $G$-module
of highest weight $\lambda$. Denote $\Pi(\lambda)$ the set of weights
occurring in $V(\lambda)$: then $\Pi(\lambda) = W \Pi^+(\lambda)$,
where
$$
	\Pi^+(\lambda) = \Pi(\lambda) \cap \mathcal{X}(T)^+ =
	\{ \mu \in \mathcal{X}(T)^+ \, : \, \mu \leqslant \lambda \}.
$$
Similarly, denote $\mathcal{P}(\lambda) \subset \mathcal{X}(T)_\mathbb{Q}$ the convex hull
of $\Pi(\lambda)$ and denote $\PQ(\lambda) = \mathcal{P}(\lambda) \cap \mathcal{X}(T)$:
then $\PQ(\lambda) = W\PQ^+(\lambda)$, where
$$
	\PQ^+(\lambda) = \mathcal{P}(\lambda) \cap \mathcal{X}(T)^+ =
	\{ \mu \in \mathcal{X}(T)^+ \, : \, \mu \leqslant_\mathbb{Q} \lambda \}.
$$

If $\Pi \subset \mathcal{X}(T)^+$, denote $V_\Pi = \bigoplus_{\mu \in \Pi} V(\mu)$ and consider the following variety
$$
	X_\Pi = \overline{(G \times G) [\Id]} \subset \mathbb{P}(\End(V_\Pi)),
$$
which is a compactification of a quotient of $G \simeq G\times G/ \diag(G\times G)$. Notice that,
if $\Id_\mu \in \End(V(\mu))$ is the identity, then we may define $X_\Pi$ as well as follows
$$
	X_\Pi = \overline{(G \times G) [(\Id_\mu)_{\mu \in \Pi}]} \subset \mathbb{P}\Big(\bigoplus_{\mu\in\Pi} \End(V(\mu))\Big):
$$
indeed $\big(V(\lambda) \otimes V(\mu)^*\big)^{\diag (G\times G)}$ is non-zero if and only if $\lambda = \mu$, in which case it is generated by $\Id_\mu$. If
$\Pi = \{\lambda_1, \ldots, \lambda_m\}$, for simplicity sometimes we will denote $X_\Pi$ also by $X_{\lambda_1, \ldots, \lambda_m}$.

\begin{dfn}
Let $\Pi \subset \mathcal{X}(T)^+$ be a set of dominant characters. We say that $X_\Pi$ is \textit{simple} if it contains a unique closed orbit. We say that $X_\Pi$ is
\textit{faithful} (resp. \textit{almost faithful}, \textit{adjoint}) if its open $G\times G$-orbit is isomorphic to $G$ (resp. to a quotient of $G$ by a finite group,
to a quotient of $G_\mathrm{ad}$). If $X_\Pi$ is simple (resp. faithful, almost faithful, adjoint), then we will say also that $\Pi$ is \textit{simple} (resp. \textit{faithful},\textit{almost faithful}, \textit{adjoint}). We say that a weight $\lambda \in \mathcal{X}(T)^+$ is \textit{almost faithful} if $\{\lambda\}$ is almost faithful, namely if
$X_\lambda$ is a compactification of $G_{\mathrm{ad}}$.
\end{dfn}

\begin{prop}[see {\cite[\S 8]{Ti}}]	\label{prop: timashev}
Let $\Pi \subset \mathcal{X}(T)^+$ be a finite subset and denote $\mathcal P(\Pi) \subset \mathcal{X}(T)_\mathbb{Q}$ the polytope generated by the $T$-weights occurring in $V_\Pi$.
\begin{itemize}
	\item[i)] $\Pi$ is faithful if and only if the set $\{\mu - \nu\}_{\mu, \nu \in \Pi}$ generates $\mathcal{X}(T)/\mathbb{Z}[\Delta]$, whereas it is adjoint if and only if $\{\mu - \nu\}_{\mu, \nu \in \Pi} \subset \mathbb{Z}[\Delta]$.
	\item[ii)] If $\mu \in \Pi$, then $X_\Pi$ contains the closed orbit of $\mathbb{P}(\End(V(\mu))$ if and only $\mu$ is a vertex of $\mathcal P(\Pi)$. In particular, $\Pi$ is simple if and only if it contains a unique maximal element w.r.t. $\leqslant_\mathbb{Q}$.
\end{itemize}
\end{prop}

In particular it follows that a simple set $\Pi$ with maximal element $\lambda$ is almost faithful if and only if $\Supp(\lambda) \cap \Delta' \neq \varnothing$ for every connected component $\Delta' \subset \Delta$. Notice also that $\Pi \subset \mathcal{X}(T)^+$ is simple and adjoint if and only if it possesses a unique maximal element w.r.t. $\leqslant$.
For instance, if $\lambda\in \mathcal{X}(T)^+$, then the set $\PQ^+(\lambda)$ is simple,
whereas $\Pi^+(\lambda)$ is simple and adjoint.
If moreover $\lambda$ is almost faithful,
then $\PQ^+(\lambda)$ is also faithful provided that the non-zero coefficients
of $\lambda$ are big enough (see Theorem \ref{teo: descrizione normalizzazione}).

We now do some recalls from the theory of the embeddings of a semisimple algebraic group, which we regard as a special case of the general theory of
spherical embeddings developed by Luna and Vust (see \cite{Kn} and \cite{Ti}). Given a semisimple group $G$, recall that a normal $G$-variety is called
\textit{spherical} if it possesses an open orbit for some Borel subgroup of $G$. We are here interested in the case of the group $G$ itself regarded as a
$G\times G$-homogeneous space: the open cell $B B^-$ is then an orbit for the Borel subgroup $B \times B^-$, so that we may regard $G$ as a spherical
$G\times G$-variety.

Let $X$ be a simple and complete normal embedding of $G$. Denote $\mathcal{D}(G)$ the set of $B\times B^{-}$-stable prime divisors of $G$ and $\mathcal{D}(X)\subset
\mathcal{D}(G)$ the set of divisors whose closure in $X$ contains the closed $G\times G$-orbit: then the Picard group $\Pic(X)$ is free and it is generated by
the classes of divisors in $\mathcal{D}(G)\smallsetminus \mathcal{D}(X)$ (see \cite[Proposition~2.2]{Br1}), while the class group $\mathrm{Cl}(X)$ is generated by the classes
of $B\times B^{-}$-stable prime divisors of $X$ (see \cite[section~2.2]{Br1}). Denote $\mathcal{N}(X)$ the set of $G\times G$-stable prime divisors of $X$, so
that the set of $B\times B^{-}$-stable prime divisors of $X$ is identified with $\mathcal{D}(G)\cup \mathcal{N}(X)$.

Notice that $\Bbbk(G)^{(B\times B^{-})}/\Bbbk^{*} \simeq \mathcal{X}(T)$ and consider the natural map $\rho : \mathcal{D}(G)\cup \mathcal{N}(X) \longrightarrow \mathcal{X}(T)^\vee$ defined by associating to a $B\times B^-$-stable prime divisor $D$ the cocharacter associated to the rational discrete valuation induced by $D$. If $D\in \mathcal{N}(X)$, then $\rho(D)$ is a negative multiple  of a fundamental coweight, while if $D\in \mathcal{D}(G)$, then $\rho(D)$ is a simple coroot; moreover $\rho$ is injective and $\rho(\mathcal{D}(G))=\Delta^\vee$ (see \cite[\S~7]{Ti}).

Denote $\mathcal{C}(X) \subset \mathcal{X}(T)^\vee_{\mathbb{Q}}$ the convex cone generated by $\rho\big(\mathcal{D}(X)\cup \mathcal{N}(X)\big)$; by the general theory of spherical embeddings we have that $\mathcal{C}(X)$ is generated by $\rho(\mathcal{D}(X))$ together with the negative Weyl chamber of $\Phi$ (see \cite[Theorem 3 and Corollary of Proposition 4]{Ti}). By definition, the \textit{colored cone} of $X$ is the couple $\big(\mathcal{C}(X),\mathcal{D}(X)\big)$: up to equivariant isomorphisms, it uniquely determines $X$ as a $G\times G$-compactification of $G$ (see \cite[Theorem 2]{Ti}).

In the case of our interest, let $\lambda \in \mathcal{X}(T)^+$ and denote $\widetilde X_\lambda \longrightarrow X_\lambda$ the normalization of $X_\lambda$ in $\Bbbk(G)$: this variety depends only on the set of simple roots $\Supp(\lambda)$ (see \cite[Proposition 1.2]{BGMR}) and we have $\rho(\mathcal{D}(\widetilde X_\lambda)) = \Delta^\vee \smallsetminus \Supp(\lambda)^\vee$ (see \cite[Theorem~7]{Ti}). When $\lambda$ is regular $\widetilde X_\lambda$ is called the \textit{canonical compactification} of $G$ and we will denote it by $M$. In case $G$ is adjoint, then $M = M_\mathrm{ad}$ is the \textit{wonderful compactification} of $G_\mathrm{ad}$ considered by De~Concini and Procesi in \cite{CP} in the more general context of adjoint symmetric varieties. By \cite{CP} we have that $M_{\mathrm{ad}}$ coincides with the variety $X_\lambda$ whenever $\lambda$ is a regular dominant weight, so that $M$ is the normalization of $M_\mathrm{ad}$ in $\Bbbk(G)$. Following the general theory of spherical varieties (see \cite[Theorem 5.1]{Kn} and \cite[Proposition 1]{Ti}), $M$ dominates every simple linear compactification of a quotient of $G$.

The closed orbit of $M$ is isomorphic to $G/B \times G/B$ and the restriction of line bundles determines an embedding of $\Pic(M)$ into $\Pic(G/B \times G/B)$, that we identify with $\Lambda \times \Lambda $, which identifies $\Pic(M)$ with the set of weights of the form $(\lambda,\lambda^*)$. Therefore $\Pic(M)$ is identified with $\Lambda$ and we denote by $\mathcal{M}_{\lambda} \in \Pic(M)$ the line bundle whose restriction to $G/B \times G/B$ is isomorphic to $\mathcal{L}_\lambda \boxtimes \mathcal{L}_{\lambda^*}$.

Since $M$ possesses an open $B \times B^-$-orbit, it follows that the $G\times G$-module $\Gamma(M,\mathcal{M}_\lambda)$ is multiplicity free. Let $\Pi \subset \mathcal{X}(T)^+$ be a simple set with maximal element $\lambda$. Then the map $G \longrightarrow \mathbb{P}(\End(V_\Pi))$ extends to a map $M \longrightarrow \mathbb{P}(\End(V_\Pi))$ whose image is $X_\Pi$. Moreover $\mathcal{M}_\lambda$ is the pullback of $\mathcal{O}(1)$, and if we pull back the homogeneous coordinates of $\mathbb{P}(\End(V_\Pi))$ to $M$ we get a submodule of $\Gamma(M,\mathcal{M}_\lambda)$ which is isomorphic to $\bigoplus_{\mu\in\Pi}\End(V(\mu))^*$. In particular, if $\mu \in \Pi^+_G(\lambda)$, then $\Gamma(M,\mathcal{M}_\lambda)$ possesses a unique submodule isomorphic to $\End(V(\mu))^*$, and by abuse of notation we will still denote it by $\End(V(\mu))^*$. Conversely, for $\lambda \in \mathcal{X}(T)^+$, if we consider the projective map associated to the complete linear system of $\mathcal{M}_\lambda$ we get a linear compactification of a quotient of $G$, say $X_\Pi$, which is simple since $M$ is so, and by Proposition \ref{prop: timashev} it follows that $\Pi$ is simple with maximal element $\lambda$. Therefore, as $G\times G$-module, the space of sections of $\mathcal{M}_\lambda$ decomposes as follows:
\[
	\Gamma(M,\mathcal{M}_\lambda)  \simeq 	\bigoplus_{\mu \in \PQ^+(\lambda)} \End(V(\mu))^*.
\]

Consider the graded algebra $\widetilde A(\lambda) = \bigoplus_{n=0}^\infty \widetilde{A}_n(\lambda)$, where $\widetilde{A}_n(\lambda)= \Gamma(M,\mathcal{M}_{n\lambda})$. By identifying sections of line bundles on $M$ with functions on $G$, it is possible to describe the multiplication in $\widetilde A(\lambda)$ in the following way.

\begin{lem}[{\cite[Lemma~3.1]{Ka} or \cite[Lemma~3.4]{DC}}] \label{lem: coefficientimatriciali}
Let $\lambda,\mu \in \mathcal{X}(T)^+$ and let $\lambda' \in \PQ^+(\lambda)$ and $\mu'\in \PQ^+(\mu)$. Then the image of $\End(V(\lambda'))^* \times \End(V(\mu'))^* \subset \Gamma(M, \mathcal{M}_{\lambda}) \times \Gamma(M, \mathcal{M}_{\mu})$ in $\Gamma(M, \mathcal{M}_{\lambda + \mu})$ via the multiplication map is
\[
	\bigoplus_{V(\nu)\subset V(\lambda')\otimes V(\mu')} \End(V(\nu))^*.
\]
\end{lem}

Denote $A(\Pi)$ the homogeneous coordinate ring of $X_\Pi \subset \mathbb{P}(\End(V_\Pi))$. As a subalgebra of $\widetilde{A}(\lambda)$, the algebra $A(\Pi)$ is generated
by $\bigoplus_{\mu \in \Pi} \End(V(\mu))^* \subset \Gamma(M,\mathcal{M}_\lambda)$ and it inherits a grading defined by $A_n(\Pi) = A(\Pi) \cap A_n(\lambda)$. Denote
$\phi_\lambda \in \End(V_\lambda)$ a highest weight vector and consider the $B\times B^-$-stable affine open subsets $X_\lambda^\circ \subset X_\lambda$ and
$X_\Pi^\circ \subset X_\Pi$ defined by the non-vanishing of $\phi_\lambda$. The coordinate rings of previous affine sets are described as follows:
$$
 \Bbbk[X_\lambda^\circ] = \left\{ \frac{\phi}{\phi_\lambda^n} \, : \, n \in \mathbb{N}, \; \phi \in A_n(\lambda) \right\}, \quad \qquad
 \Bbbk[X_\Pi^\circ] = \left\{ \frac{\phi}{\phi_\lambda^n} \, : \, n \in \mathbb{N}, \; \phi \in A_n(\Pi) \right\}.
$$
and the natural morphism $X_\Pi \longrightarrow X_\lambda$ induces a natural inclusion $\Bbbk[X_\lambda^\circ] \subset \Bbbk[X_\Pi^\circ]$. Previous rings are not $G\times
G$-modules, however they are $\mathfrak{g}\oplus \mathfrak{g}$-modules.

Following lemma and proposition were given in \cite{Ga} in the case of a simple linear adjoint compactification, however their proofs generalize
straightforwardly to the non-adjoint case.

\begin{lem} [{\cite[Lemma~1.3]{Ga}}] \label{lem: generatori X_Pi}
Let $\Pi\subset \mathcal{X}(T)^+$ be simple with maximal element $\lambda$.
As a $\mathfrak{g} \oplus \mathfrak{g}$-algebra, $\Bbbk[X^\circ_\Pi]$ is generated by $\Bbbk[X^\circ_{\lambda}]$ together with the set $\{\phi_\mu/\phi_\lambda\}_{\mu \in \Pi}$.
\end{lem}

\begin{prop}[{\cite[Prop.~1.6 and Cor.~1.10]{Ga}}]   \label{prop: criterio tensoriale morfismi}
Let $\Pi, \Pi'\subset \mathcal{X}(T)^+$ be simple subsets with maximal elements
resp. $\lambda$ and $\lambda'$ and suppose that $\Supp(\lambda) = \Supp(\lambda')$.
\begin{itemize}
	\item[i)] There exists a $G \times G$-equivariant morphism $X_\Pi \longrightarrow X_{\Pi'}$ if and only if, for every $\mu' \in \Pi'$, there exist $n \in \mathbb{N}$
and $\mu_1, \ldots, \mu_n \in \Pi$ such that $V\big(\mu' - \lambda' + n \lambda \big) \subset V(\mu_1) \otimes \ldots \otimes V(\mu_n)$.
	\item[ii)]  If $\{\mu - \lambda \}_{\mu \in \Pi} = \{\mu' - \lambda'\}_{\mu' \in \Pi'}$,
then $X_\Pi \simeq X_{\Pi'}$ as $G \times G$-varieties.
\end{itemize}
\end{prop}

We now recall the criterion of normality for a simple
adjoint compactification given in \cite{BGMR}.
Let $\Pi \subset \mathcal{X}(T)^+$ be simple and adjoint with maximal element $\lambda$.
Then $\Pi \subset \Pi^+(\lambda)$ and we have a natural projection
$X_\Pi^+(\lambda) \longrightarrow X_\Pi$.
While Kannan shown in \cite{Ka} that $X_\Pi^+(\lambda)$ is a projectively normal variety,
De~Concini proved in \cite{DC} that $X_{\Pi^+(\lambda)} \longrightarrow X_\Pi$
is the normalization of $X_\Pi$.
Hence we deduce by Lemma \ref{prop: criterio tensoriale morfismi}
the following characterization of the normality of $X_\Pi$.

\begin{prop}[{\cite[Prop. 2.3]{BGMR}}]	\label{prop: criterio tensoriale normalità aggiunto}
Let $\Pi \subset \mathcal{X}(T)^+$ be simple and adjoint with maximal element $\lambda$.
Then the variety $X_\Pi$ is normal if and only if for every $\mu \in \Pi^+(\lambda)$
they exist $n \in \mathbb{N}$ and $\lambda_1,\ldots,\lambda_n \in \Pi$ such that
\[ V(\mu + (n-1)\lambda) \subset V(\lambda_1) \otimes \cdots \otimes V(\lambda_n). \]
\end{prop}

The following definition allows to make effective previous characterization.

\begin{dfn} \label{def: twin}
If $\Delta' \subset \Delta$ is a non-simply laced connected component, order the simple roots in $\Delta'= \{ \alpha_1, \ldots, \alpha_r\}$ starting from the
extremal long root and denote $\alpha_q$ the first short root in $\Delta'$. Let $\lambda \in \Lambda^+$ be such that $\alpha_q\not \in \Supp(\lambda)$ and such that
$\Supp(\lambda)\cap \Delta'$ contains a long root, denote $\alpha_p$ the last long root which occurs in $\Supp(\lambda)\cap \Delta'$. For instance, if $\Delta'$ is not
of type $\mathsf{G}_2$, then the numbering is as follows:
\[
\begin{picture}(9000,1800)(2000,-900)
           \put(0,0){\multiput(0,0)(3600,0){2}{\circle*{150}}\thicklines\multiput(0,0)(2500,0){2}{\line(1,0){1100}}\multiput(1300,0)(400,0){3}{\line(1,0){200}}}
           \put(3600,0){\multiput(0,0)(3600,0){2}{\circle*{150}}\thicklines\multiput(0,0)(2500,0){2}{\line(1,0){1100}}\multiput(1300,0)(400,0){3}{\line(1,0){200}}}
           \put(7200,0){\multiput(0,0)(1800,0){2}{\circle*{150}}\thicklines\multiput(0,-60)(0,150){2}{\line(1,0){1800}}\multiput(1050,0)(-25,25){10}{\circle*{50}}\multiput(1050,0)(-25,-25){10}{\circle*{50}}}
           \put(9000,0){\multiput(0,0)(3600,0){2}{\circle*{150}}\thicklines\multiput(0,0)(2500,0){2}{\line(1,0){1100}}\multiput(1300,0)(400,0){3}{\line(1,0){200}}}
           \put(-150,-700){\tiny $\alpha_1$}
           \put(3450,-700){\tiny $\alpha_p$}
           \put(8850,-700){\tiny $\alpha_q$}
           \put(12450,-700){\tiny $\alpha_r$}
\end{picture}
\]
The \textit{adjoint little brother} of $\lambda$ with respect to $\Delta'$ is the dominant weight
\[
\lambda_{\Delta'}^\mathrm{lb} = \lambda - \sum_{i=p}^q \alpha_i =
\left\{ \begin{array}{ll}
		\lambda-\omega_1+\omega_2 & \textrm{ if $G$ is of type $\sf{G}_2$} \\
		\lambda + \omega_{p-1} - \omega_{p}  + \omega_{q+1} & \textrm{ otherwise}
\end{array} \right.
\]
where $\omega_i$ is the fundamental weight associated to $\alpha_i$ if $1\leq i \leqslant r$, while $\omega_0 = \omega_{r+1} = 0$.
\end{dfn}

The set of adjoint little brothers of
$\lambda$ will be denoted by $\LB{\lambda}$,
while we set $\LBb{\lambda} = \LB{\lambda} \cup \{\lambda\}$.
Notice that if $G$ is simply laced then $\LB{\lambda} = \varnothing$ for every $\lambda \in \mathcal{X}(T)^+$.

\begin{teo}[{\cite[Thm.~2.10]{BGMR}}]	\label{teo: normalita caso aggiunto}
Let $\Pi \subset \mathcal{X}(T)^+$ be simple and adjoint with maximal element $\lambda$.
Then the variety $X_\Pi$ is normal if and only if $\Pi \supset \LB{\lambda}$.
\end{teo}

In particular, it follows from the previous theorem that, if $G$ is simply laced, then every simple adjoint compactification is normal.\\

If $n \in \mathbb{N}$ consider the set
\[	\mathrm{Tens}_n(G) = \{(\lambda_0, \ldots, \lambda_n) \in \mathcal{X}(T)^+ \times \ldots \times \mathcal{X}(T)^+ \, : \, V(\lambda_0) \subset V(\lambda_1) \otimes
\ldots \otimes V(\lambda_n)\}.	\]
Following lemma has been proved in several references, usually in the case $n=2$. For later use, we state it in a slightly more general form, which is
easily reduced to the case $n=2$ proceeding by induction on $n$.

\begin{lem}	[{\cite[Lemma 3.9]{Ku2}}] \label{lem: tensor semigroup}
The set $\mathrm{Tens}_n(G)$ is a semigroup with respect to the addition.
\end{lem}

An easy consequence of previous lemma is the following.

\begin{cor} \label{cor: traslazione}
Let $\lambda, \mu, \nu \in \mathcal{X}(T)^+$ be such that $V(\nu) \subset V(\lambda)
\otimes V(\mu)$. Then, for any $\nu' \in \mathcal{X}(T)^+$, it also holds
$V(\nu + \nu') \subset V(\lambda + \nu') \otimes V(\mu)$.
\end{cor}

When $G$ is of type $\mathsf{A}$, the following saturation property of $\mathrm{Tens}_n(G)$ holds:
$$
\begin{array}{l}
\text{If } N > 0 \text{ and } \lambda_0, \ldots, \lambda_n \in \mathcal{X}(T)^+ \text{ are such that } (N\lambda_0, \ldots, N\lambda_n) \in \mathrm{Tens}_n(G),\\
\text{ then } (\lambda_0, \ldots, \lambda_n) \in \mathrm{Tens}_n(G) \text{ as well.}
\end{array}
$$
In case $n=2$, previous property was proved by Knutson and Tao in \cite{KT}, and then it has been conjectured for every simply laced group by Kapovich and Millson in \cite{KM}. As discussed in the survey of Fulton \cite{Fu}, when $G$ is of type $\mathsf{A}$ the case of a general $n$ can be deduced from the case $n=2$.

\begin{oss}
We already noticed that Theorem \ref{teo: normalita caso aggiunto} implies in particular that, if $G$ is simply laced, then every simple adjoint compactification is normal. When $G$ is of type $\mathsf{A}$, a very easy proof of this fact follows by the saturation of $\mathrm{Tens}_n(G)$ thanks to Proposition \ref{prop: criterio tensoriale normalità aggiunto}. Let indeed $\lambda \in \mathcal{X}(T)^+$ and $\mu\in \Pi^+(\lambda)$. Then there exists $n \in \mathbb{N}$ such that $V(n\mu) \subset V(\lambda)^{\otimes n}$ (see \cite[Lemma 4.9]{AB} or \cite[Lemma 1]{Ti}). By Corollary \ref{cor: traslazione}, this implies that $V(n\mu + n(n-1)\lambda) \subset V(n\lambda)^{\otimes n}$, hence the saturation property implies that $V(\mu + (n-1)\lambda) \subset V(\lambda)^{\otimes n}$.
\end{oss}

\section{The partial orders $\leqslant^\lambda$ and $\leqslant^\lambda_\mathbb{Q}$.}

Let $\lambda \in \Lambda^+$. This section is devoted to proving some properties of the partial orderings $\leqslant^\lambda$ and $\leqslant^\lambda_\mathbb{Q}$ that we defined in the
introduction. Notice that if $\lambda, \lambda' \in \Lambda^+$ are such that $\Supp(\lambda) = \Supp(\lambda')$, then $\lambda$ and $\lambda'$ induce the same partial orders on
$\Lambda$.

\begin{prop} \label{prop: lambda dominanza in natura}
If $\lambda \in \mathcal{X}(T)^+$ and $\pi \in \Pi(\lambda)$, then $\pi \leqslant^\lambda \lambda$.
\end{prop}

\begin{proof}
Let $v_\lambda, v_\lambda^- \in V(\lambda)$ be respectively a highest weight vector and a lowest weight vector and denote $P,P^-$ their stabilizers in $G$. Denote
$\mathfrak{p}$ the Lie algebra of $P$ and $\mathfrak{n}^-$ the Lie algebra of the unipotent radical of $P^-$ and let $\mathfrak{U}(\mathfrak{g})$ (resp. $\mathfrak{U}(\mathfrak{n}^-)$) be the universal
enveloping algebra of $\mathfrak{g}$ (resp. of $\mathfrak{n}^-$). Then $\mathfrak{g} = \mathfrak{n}^- \oplus \mathfrak{p}$ and it follows $V(\lambda) = \mathfrak{U}(\mathfrak{g}) v_\lambda = \mathfrak{U}(\mathfrak{n}^-) v_\lambda$. On
the other hand, if $\mathfrak{g}_\alpha \subset \mathfrak{g}$ denotes the root space of $\alpha \in \Phi$, then $\mathfrak{n}^- = \bigoplus_{\alpha \in \Phi^-(\lambda)} \mathfrak{g}_\alpha$, hence
by looking at $T$-weights it follows $\pi - \lambda \in \mathbb{N}[\Phi^-(\lambda)]$.
\end{proof}

\begin{cor} \label{cor: lambda dominanza in natura}
Let $\lambda \in \mathcal{X}(T)^+$ and let $\lambda_1,\ldots,\lambda_n \in \Pi^+(\lambda)$.
If $\mu, \nu \in \mathcal{X}(T)^+$ are such that
$$
	V(\nu) \subset V(\mu)  \otimes V(\lambda_1) \otimes \ldots \otimes V(\lambda_n),
$$
then $\nu \leqslant^\lambda \mu + n \lambda$.
\end{cor}

\begin{proof}
We proceed by induction on $n$, suppose first that $n=1$. Recall that every highest weight in $V(\mu) \otimes V(\lambda_1)$ is of the shape $\mu + \pi$ for
some $\pi \in \Pi(\lambda_1)$ (see for instance \cite[Proposition 3.2]{Ku2}). On the other hand $\Pi(\lambda_1) \subset \Pi(\lambda)$, therefore previous
proposition implies that $\nu\leq^\lambda \mu + \lambda$. Suppose now $n > 1$ and let $\mu' \in \mathcal{X}(T)^+$ be such that
$$
V(\mu') \subset V(\mu) \otimes V(\lambda_1) \otimes \ldots \otimes V(\lambda_{n-1})
\qquad \text{ and } \qquad V(\nu) \subset V(\mu') \otimes V(\lambda_n):
$$
Then by the inductive hypothesis we get $\mu' \leqslant^\lambda \mu + (n-1) \lambda$ and $\nu\leq^\lambda \mu' + \lambda$ and the claim follows.
\end{proof}

Following proposition will be the core of the necessity part of Theorem 1.

\begin{prop} \label{prop: necessita}
Suppose that $\nu, \mu_1, \ldots, \mu_n \in \PQ^+(\lambda)$ are such that
\[ V(\nu + (n-1)\lambda) \subset V(\mu_1) \otimes \cdots \otimes V(\mu_n). \]
Then $\nu \leqslant_\mathbb{Q}^\lambda \mu_i$ for every $i$.
\end{prop}

\begin{proof}
We only show the inequality $\nu \leqslant_\mathbb{Q}^\lambda \mu_1$, the others are analogous.
Set $c = \det(C)$, where $C$ is the Cartan matrix of $\Phi$,
and denote $\pi = \nu + (n-1)\lambda$: we will prove the claim by showing that
$$c \pi \leqslant^\lambda c\mu_1 + (n-1)c\lambda.$$

By Theorem \ref{teo: normalita caso aggiunto} the variety $X_{\LBb{c\lambda}}$ is normal.
Set $\LBb{c\lambda} = \{\lambda_1,\ldots,\lambda_s\}$ and let $i>1$.
Since $c \mu_i \leqslant c \lambda$, by Proposition
\ref{prop: criterio tensoriale normalità aggiunto}
they exist $M_i^1, \ldots, M_i^s \in \mathbb{N}$ such that
$$
	V(c\mu_i + (N_i-1) c \lambda) \subset \bigotimes_{j=1}^sV(c \lambda_j)^{\otimes M_i^j}
$$
where we set $N_i = \sum_{j=1}^s M_i^j$. On the other hand,
the hypothesis together with Lemma \ref{lem: tensor semigroup} imply that
$$
	V(c \pi) \subset V(c \mu_1) \otimes \cdots \otimes V(c \mu_n).
$$
Combining previous inclusions together with Corollary \ref{cor: traslazione} we get then
$$
	V\big(c \pi + \sum_{i=2}^n(N_i-1) c \lambda\big) \subset
	V(c \mu_1) \otimes \bigotimes_{j=1}^sV(c \lambda_j)^{\otimes M_2^j + \ldots + M_n^j}.
$$

Since the partial orderings induced by $\lambda$ and by $c \lambda$ are the same,
Proposition \ref{prop: lambda dominanza in natura} together
with Corollary \ref{cor: lambda dominanza in natura} show that
\[
	c \pi + \sum_{i=2}^n(N_i-1) c \lambda \leqslant^\lambda c \mu_1 +
\sum_{i=2}^n N_i  c\lambda,
\]
i.e. $c \pi \leqslant^\lambda c \mu_1 + (n-1) c \lambda$.
\end{proof}

Denote $\mathcal{C}(\lambda) \subset \mathcal{X}(T)_\mathbb{Q}$ the cone generated by $\Phi^-(\lambda)$
and recall that $\mathcal{P}(\lambda)$ is the polytope generated in $\mathcal{X}(T)_\mathbb{Q}$ by the orbit $W\lambda$.
It follows from the definition that, if $\nu, \mu \in \mathcal{X}(T)$,
then $\nu \leqslant^\lambda_\mathbb{Q} \mu$ if and only if $\nu - \mu \in \mathcal{C}(\lambda)$.

\begin{prop}	\label{prop: lambda coni}
The cone $\mathcal{C}(\lambda)$ is generated by $\mathcal{P}(\lambda) - \lambda$.
Moreover, the following equality holds:
$$
	\mathcal{C}(\lambda) = \bigcup_{n\in \mathbb{N}} \mathcal{P}(n\lambda) - n\lambda.
$$
\end{prop}

\begin{proof}
If $\pi_1, \pi_2 \in \mathcal{X}(T)$, denote by $\overline{\pi_1, \pi_2} \subset \mathcal{X}(T)_\mathbb{Q}$
the line segment connecting $\pi_1$ and $\pi_2$.
Since $\mathcal{P}(\lambda)$ is a convex polytope, the cone generated by $\mathcal{P}(\lambda) - \lambda$
is generated by the set
$$
	\mathcal{G}(\lambda) = \{w\lambda - \lambda \, : \, \overline{w\lambda,\lambda} \text{ is an edge of } \mathcal{P}(\lambda)\}.$$
By the following lemma, every element in $\mathcal{G}(\lambda)$ is a positive multiple
of an element in $\Phi^-(\lambda)$. Therefore $\mathcal{P}(\lambda) - \lambda$ generates $\mathcal{C}(\lambda)$
and the second claim is trivial.
\end{proof}

\begin{lem}
Let $w \in W$ be such that $\overline{w\lambda,\lambda}$ is an edge of $\mathcal{P}(\lambda)$. Then $w = s_\tau$ is the reflection associated to a positive root $\tau \in \Phi^+(\lambda)$.
\end{lem}

\begin{proof}
Denote $\mathcal{C}$ the cone generated by $\mathcal{P}(\lambda) - \lambda$.
Notice that $\Phi^-(\lambda) \subset \mathcal{C}$: indeed if $\sigma \in \Phi^-(\lambda)$
then $s_\sigma\lambda \in\mathcal{P}(\lambda)$ and $s_\sigma \lambda - \lambda$ is a non-zero multiple of $\sigma$.

By hypothesis, the vector $w\lambda-\lambda$ generates an extremal ray of $\mathcal{C}$,
while by Proposition \ref{prop: lambda dominanza in natura} we may write
$w\lambda-\lambda = \sum_{\sigma \in \Phi^-(\lambda)} a_\sigma\sigma$ with $a_\sigma \geqslant 0$ for all $\sigma$.
Since $\Phi^-(\lambda) \subset \mathcal{C}$, it follows that there exists only one
$\tau\in\Phi^-(\lambda)$ with $a_\tau \neq 0$.
In particular, $w\lambda - \lambda$ is a multiple of $\tau$,
hence $s_\tau \lambda$ coincides with $w\lambda$ because they are both vertices of $\mathcal{P}(\lambda)$.
\end{proof}

If $\alpha \in \Delta$, denote $\Delta(\alpha) \subset \Delta$ the connected component
containing $\alpha$ and denote $r_\alpha$ the rank of $\Delta(\alpha)$.
If $\beta \in \Delta(\alpha)$, denote $d(\alpha,\beta) \in \mathbb{N}$
the distance between $\alpha$ and $\beta$ in the Dynkin diagram of $\Delta(\alpha)$.
Denote by $I(\alpha, \beta) \subset \Delta$ the minimal connected subset
containing both $\alpha$ and $\beta$, so that $|I(\alpha, \beta)| = d(\alpha, \beta)+1$.
We now describe a property of the maximal elements in $\PQ^+(\lambda)$ w.r.t. $\leqslant^\lambda$
which will be one of the main tools in the characterization of the normality of a simple linear compactification of $G$. First we need an auxiliary lemma.

\begin{lem}	\label{lem: coefficienti mu}
Let $\lambda \in \mathcal{X}(T)^+$ and let $\mu \in \PQ^+(\lambda)$.
Let $\alpha \in \Supp(\lambda)$ be such that
$\langle \lambda - \mu, \alpha^\vee \rangle > r_\alpha$
and denote $S(\alpha) \subset \big(\Delta(\alpha) \smallsetminus \Supp(\mu)\big) \cup \{\alpha\}$
the connected component of $\alpha$.
Fix $\beta$ an extremal root of $S(\alpha)$, set $I(\alpha,\beta) = \{\gamma_0, \gamma_1, \ldots, \gamma_k\}$
(where $k = d(\alpha,\beta)$ and where the numbering is defined by $\gamma_0 = \alpha$ and
$\langle \gamma_i, \gamma_{i-1}^\vee \rangle \neq 0$ for every $i = 1, \ldots, k$)
and denote $(\lambda - \mu )\bigr|_{I(\alpha, \beta)} = \sum_{i=0}^k c_i \gamma_i$.
\begin{itemize}
	\item[i)] 	Set $c_{k+1} = 0$. For every $i \leqslant k$, it holds
	$$
		c_i \geqslant \frac{(i+1) c_{i+1} + r_\alpha + 1}{i+2}
	$$
	\item[ii)] For every $i\leq k$ it holds
	$$
		c_i \geqslant k-i+1 + \frac{r_\alpha - k -1}{i+2}
	$$
\end{itemize}
\end{lem}

\begin{proof}
i) We prove the claim by induction on $i$.
The case $i = 0$ follows immediately by the inequality
$$
	r_\alpha +1  \leqslant \langle \lambda - \mu, \alpha^\vee \rangle \leqslant 2c_0 - c_1,
$$
so we assume $i >0$. Notice that
$$
	\langle \lambda , \gamma_i^\vee \rangle + c_{i-1} - 2 c_i + c_{i+1} \leqslant \langle \mu, \gamma_i^\vee \rangle = 0:
$$
hence by the inductive hypothesis we get
$$
	2c_i \geqslant c_{i-1} + c_{i+1} \geqslant \frac{i c_i + r_\alpha + 1}{i+1} + c_{i+1}
$$
and the inequality follows.

ii) We prove the claim by induction on $k-i$. If $i=k$, then by part i) we get
$$
	c_k \geqslant \frac {r_\alpha + 1}{k+2} =
	1 + \frac{r_\alpha - k -1}{k+2}.
$$

Assume $i < k$. Then by part i) together with the hypotheses we get
\[
	c_i \geqslant
	\frac{(i+1) c_{i+1} + r_\alpha + 1}{i+2}
	\geqslant \frac{(i+1) (k-i) + r_\alpha + 1}{i+2} =
	k-i+1 + \frac{r_\alpha - k - 1}{i+2}.	\qedhere
\]
\end{proof}

\begin{table}
\begin{center}
\begin{tabular}{c|c}
type of $\Phi$ & highest short root \\
\hline $\mathsf{A}_r$ & $\alpha_1+\cdots+\alpha_r$ \\
$\mathsf{B}_r$ & $\alpha_1+\cdots+\alpha_r$ \\
$\mathsf{C}_r$ & $\alpha_1+2(\alpha_2+\cdots+\alpha_{r-1})+\alpha_r$ \\
$\mathsf{D}_r$ & $\alpha_1+2(\alpha_2+\cdots+\alpha_{r-2})+\alpha_{r-1}+\alpha_r$ \\
$\mathsf{E}_6$ & $\alpha_1+2\alpha_2+2\alpha_3+3\alpha_4+2\alpha_5+\alpha_6$ \\
$\mathsf{E}_7$ & $2\alpha_1+2\alpha_2+3\alpha_3+4\alpha_4+3\alpha_5+2\alpha_6+\alpha_7$ \\
$\mathsf{E}_8$ & $2\alpha_1+3\alpha_2+4\alpha_3+6\alpha_4+5\alpha_5+4\alpha_6+3\alpha_7+2\alpha_8$ \\
$\mathsf{F}_4$ & $\alpha_1+2\alpha_2+3\alpha_3+2\alpha_4$ \\
$\mathsf{G}_2$ & $2\alpha_1+\alpha_2$ \\
\hline
\end{tabular}
\end{center}
\caption{Highest short roots} \label{tab: highest short roots}
\end{table}

\begin{prop}		\label{prop: elementi massimali}
Let $\lambda\in \mathcal{X}(T)^+$ and let $\mu \in \PQ^+(\lambda)$ be maximal w.r.t. $\leqslant^\lambda$. Then $\langle \lambda - \mu , \alpha^\vee \rangle \leqslant r_\alpha$ for every
$\alpha \in \Delta$.
\end{prop}

\begin{proof}
We will prove the proposition by showing that,
if $\alpha \in \Supp(\lambda)$ is such that
$\langle \lambda - \mu, \alpha^\vee \rangle > r_\alpha$,
then there exists a weight $\mu' \in \PQ^+(\lambda)$
such that $\mu' - \mu \in \Phi^+(\lambda)$, contradicting the maximality of $\mu$.
Fix such a root $\alpha$, denote
$S(\alpha) \subset \big(\Delta(\alpha) \smallsetminus \Supp(\mu)\big) \cup \{\alpha\}$
its connected component and denote $\theta_\alpha \in \mathbb{N}[S(\alpha)]$
the highest short root of the root subsystem associated to $S(\alpha)$.

Suppose that the root lattice $\mathbb{Z}[\Delta(\alpha)]$ equals
the weight lattice of $\Delta(\alpha)$.
Then by Proposition \ref{prop: lambda dominanza in natura} it follows
$\mu \bigr|_{\Delta(\alpha)} \leqslant^\lambda \lambda \bigr|_{\Delta(\alpha)}$,
so that the maximality of $\mu$ implies the equality
$\lambda \bigr|_{\Delta(\alpha)} = \mu \bigr|_{\Delta(\alpha)}$
and no root such as $\alpha$ can exist.
In particular, this excludes that $\Delta(\alpha)$ is of type $\mathsf{E}_8$, $\mathsf{F}_4$ or $\mathsf{G}_2$.

Set $r = r_\alpha$ and denote $\Delta(\alpha) = \{\alpha_1, \ldots, \alpha_r\}$, where the numbering is that of \cite{Bo}, and set
$$
	(\lambda - \mu )\bigr|_{\Delta(\alpha)}= \sum_{i=1}^r a_i \alpha_i.
$$
We claim that $\theta_\alpha \leqslant_\mathbb{Q} \lambda - \mu$. Let indeed $\alpha_i \in S(\alpha)$ and denote
$$
	d(\alpha_i) = \; \min \{d(\alpha_i,\beta) \, : \, \beta \text{ is an extremal root of } S(\alpha)\}.
$$
If $\beta$ is an extremal root of $S(\alpha)$ such that $\alpha_i \in I(\alpha,\beta)$, then Lemma \ref{lem: coefficienti mu} ii) implies that
\begin{equation} \label{eqn: elementi massimali}
	a_i \geqslant d(\alpha_i,\beta) + 1 + \frac{r - d(\alpha,\beta) -1}{d(\alpha,\alpha_i) + 1}.
\end{equation}
In particular we get $a_i \geqslant d(\alpha_i) +1$, which implies that
$a_i \geqslant 1$ for every $\alpha_i \in S(\alpha)$ and
$a_i \geqslant 2$ for every $\alpha_i \in S(\alpha)$
which is not extremal in $S(\alpha)$.
Therefore by Table \ref{tab: highest short roots} it follows that
$\theta_\alpha \leqslant_\mathbb{Q} \lambda - \mu$ whenever $S(\alpha)$ is of type $\mathsf{A}$,$\mathsf{B}$,$\mathsf{C}$ or $\mathsf{D}$.

Suppose that $S(\alpha)$ is of type $\mathsf{E}$. Then
$d(\alpha,\beta) +1 < r$ for every extremal root $\beta \in S(\alpha)$, therefore
inequality (\ref{eqn: elementi massimali}) implies that
\begin{equation} \label{eqn: elementi massimali 2}
a_i > d(\alpha_i,\beta) + 1
\end{equation}
for every $\alpha_i \in S(\alpha)$. Suppose that $S(\alpha)$ is of type $\mathsf{E}_6$:
then by the description of the fundamental weights
we get that $a_2,a_4 \in \mathbb{Z}$,
hence inequality (\ref{eqn: elementi massimali 2})
implies $a_2 \geqslant 2$ and $a_4 \geqslant 3$,
and by Table \ref{tab: highest short roots}
it follows $\theta_\alpha \leqslant_\mathbb{Q} \lambda - \mu$.

Suppose finally that $S(\alpha)$ is of type $\mathsf{E}_7$.
Then by the description of the fundamental weights
we get that $a_1, a_3, a_4 \in \mathbb{Z}$ and $a_2 \in \frac{1}{2}\mathbb{Z}$,
hence inequality (\ref{eqn: elementi massimali 2}) implies
$a_1 \geqslant 2$, $a_2 \geqslant 3/2$, $a_3 \geqslant 3$, $a_4 \geqslant 3$.
Therefore it follows from Table \ref{tab: highest short roots} that
$\theta_\alpha \leqslant_\mathbb{Q} \lambda - \mu$ unless $a_2 = 3/2$ or $a_4=3$.
To show that $a_4 > 3$, notice that we may choose
the extremal root $\beta \in S(\alpha)$ in such a way that $\alpha_4 \in I(\alpha,\beta)$
and $d(\alpha_4,\beta) \geqslant 2$: therefore inequality
(\ref{eqn: elementi massimali 2}) implies $a_4 > 3$.
To show that $a_2 >3/2$, suppose first that $\alpha \neq \alpha_2$:
since $\Supp(\mu) \cap S(\alpha) \subset \{\alpha\}$, it follows that
$$
	\langle \lambda , \alpha_2^\vee \rangle - 2a_2 + a_4 = \langle \mu, \alpha_2^\vee \rangle = 0,
$$
whence $a_2 \geqslant \frac{1}{2}a_4 \geqslant 2$. Suppose now that $\alpha = \alpha_2$:
then we may choose as extremal root $\beta = \alpha_7$
and inequality (\ref{eqn: elementi massimali 2})
implies that $a_2 > d(\alpha_2,\alpha_7) +1 = 5$.

We proved so far that $\theta_\alpha \leqslant_\mathbb{Q} \lambda - \mu$.
Since $\mu$ is maximal w.r.t $\leqslant^\lambda$,
the weight $\mu + \theta_\alpha$ must be non-dominant
and an easy case-by-case consideration shows that $\Delta(\alpha)$
is of type $\mathsf{B}_r$, that $S(\alpha) = \{\alpha_p, \ldots, \alpha_{r-1}\}$ for some $p < r$
and that $\langle \mu , \alpha_r^\vee \rangle = 1$, and in particular it follows
$$
	\langle \lambda,\alpha_r^\vee \rangle + 2a_{r-1} - 2a_r = \langle \mu,\alpha_r^\vee\rangle = 1.
$$
On the other hand inequality (\ref{eqn: elementi massimali}) implies $a_{r-1} > 1$,
and being $a_{r-1} \in \frac{1}{2}\mathbb{Z}$ it follows that $a_{r-1} \geqslant 3/2$.
Thus we get $2a_r \geqslant 2a_{r-1} - 1 \geqslant 2$
and it follows that $\theta_\alpha + \alpha_r \leqslant_\mathbb{Q} \lambda - \mu$:
hence $\mu + \theta_\alpha + \alpha_r \in \PQ^+(\lambda)$,
contradicting the maximality of $\mu$.
\end{proof}

\begin{cor}	\label{cor: elementi massimali}
Let $\lambda\in \mathcal{X}(T)^+$ be such that $\langle \lambda, \alpha^\vee \rangle \geqslant r_\alpha$ for every $\alpha \in \Supp(\lambda) \smallsetminus \Supp(\lambda)^\circ$ and let $\mu
\in \PQ^+(\lambda)$ be maximal w.r.t. $\leqslant^\lambda$. If $\alpha \in \Supp(\lambda)^\circ$, then $\langle \lambda - \mu , \alpha^\vee \rangle \leqslant 1$.
\end{cor}

\begin{proof}
Suppose that $\mu \in \PQ^+(\lambda)$ is maximal w.r.t. $\leqslant^\lambda$ and denote $\lambda - \mu = \sum_{\alpha \in \Delta} c_\alpha \alpha$. Let $\beta \in
\Supp(\lambda)^\circ$ be such that $\langle \lambda - \mu , \beta^\vee \rangle \geqslant 2$: since $\langle \lambda - \mu, \beta^\vee \rangle \leqslant 2c_\beta$, it follows
then $c_\beta \geqslant 1$. Since $\mu$ is maximal, the weight $\mu + \beta$ is not dominant, on the other hand since $\beta \in \Supp(\lambda)^\circ$ there exists
$n>1$ such that the weight $\mu + (n-1) \lambda + \beta$ is dominant. Denote $\lambda' = n \lambda$ and $\mu' = \mu + (n-1) \lambda$ and let $\nu' \in \PQ^+(\lambda)$ be
maximal w.r.t. $\leqslant^\lambda$ such that $\mu' <^\lambda \nu' \leqslant_\mathbb{Q} \lambda'$. Since $\lambda' - \nu' < \lambda - \mu$, proceeding inductively we may assume that,
taking $n$ big enough, the inequality $\langle \lambda' - \nu', \alpha^\vee \rangle \leqslant 1$ holds for every $\alpha \in \Supp(\lambda)^\circ$. Since $\langle \lambda,
\alpha^\vee \rangle \geqslant r_\alpha$ for every $\alpha \in \Supp(\lambda) \smallsetminus \Supp(\lambda)^\circ$, by Proposition \ref{prop: elementi massimali} it follows then
that the weight $\nu = \lambda + \nu' - \lambda'$ is dominant, hence $\mu <^\lambda \nu \leqslant^\lambda_\mathbb{Q} \lambda$ and we get a contradiction since $\mu$ was supposed to
be maximal.
\end{proof}

\begin{dfn}
Define the following sets:
\begin{align*}
	\PQm{\lambda} &= \left\{ \right. \mu \in \PQ^+(\lambda) \, : \, \mu \text{ is maximal w.r.t. } \leqslant^\lambda \left. \right\} \cup \LB{\lambda}, \\
	\LBQ{\lambda} =  \left\{ \right. &\mu \in (\PQ^+(\lambda) \smallsetminus \Pi^+(\lambda)) \cup \LB{\lambda} \, : \, \mu \text{ is maximal w.r.t. } \leqslant^\lambda_\mathbb{Q} \left.
\right\}.
\end{align*}
We call $\LBQ{\lambda}$ the set of \textit{little brothers} of $\lambda$ \textit{w.r.t.} $G$ and we set $\LBQb{\lambda} = \LBQ{\lambda} \cup \{\lambda\}$.
\end{dfn}

Define the following sets:
$$
	\mathcal{H}^{\mathrm{m}}_G(\lambda) = \left\{\mu - \lambda \, : \, \mu \in \PQm{\lambda} \right\},
	\qquad \quad \qquad
	\mathcal{H}_G(\lambda) = \left\{\mu - \lambda \, : \, \mu \in \LBQ{\lambda} \right\}.
$$

\begin{cor} \label{cor: elementi massimali 2}
If $\lambda \in \mathcal{X}(T)^+$, set $\overline\lambda = \sum_{\alpha\in \Supp(\lambda)} n_\alpha \omega_\alpha$, where
$$
	n_\alpha = \left\{ \begin{array}{ll}
						1	& \text{ if } \alpha \in \Supp(\lambda)^\circ \\
						r_\alpha & \text{ if } \alpha \in \Supp(\lambda) \smallsetminus \Supp(\lambda)^\circ
			\end{array} \right.
$$
If $\langle \lambda, \alpha^\vee \rangle \geqslant r_\alpha$ for every $\alpha \in \Supp(\lambda)\smallsetminus \Supp(\lambda)^\circ$,
then $\mathcal{H}^{\mathrm{m}}_G(\lambda) = \mathcal{H}^{\mathrm{m}}_G(\overline\lambda)$ and $\mathcal{H}_G(\lambda) = \mathcal{H}_G(\overline\lambda)$.
\end{cor}

\begin{proof}
Let $\lambda \in \mathcal{X}(T)^+$ be such that $\langle \lambda, \alpha^\vee \rangle \geqslant r_\alpha$ for every $\alpha \in \Supp(\lambda) \smallsetminus \Supp(\lambda)^\circ$. If $\mu \in
\PQ^+(\lambda)$ is maximal w.r.t. $\leqslant^\lambda$, then by Proposition \ref{prop: elementi massimali} together with  Corollary \ref{cor: elementi massimali} it
follows that $\overline\lambda + \mu -\lambda \in \PQ^+(\overline\lambda)$, and its maximality w.r.t. $\leqslant^\lambda$ follows by the maximality of $\mu$. Similarly, if $\overline\mu \in
\PQ^+(\overline\lambda)$ is maximal w.r.t. $\leqslant^\lambda$, then $\mu = \overline\mu + \lambda - \overline\lambda \in \PQ^+(\lambda)$ is also maximal w.r.t. $\leqslant^\lambda$. Hence we get
$\mathcal{H}^{\mathrm{m}}_G(\lambda) = \mathcal{H}^{\mathrm{m}}_G(\overline\lambda)$, and the equality $\mathcal{H}_G(\lambda) = \mathcal{H}_G(\overline\lambda)$ follows straightforwardly.
\end{proof}

\begin{ex} \label{ex: omega1 in Ar}
Let $G = \mathrm{SL}(r+1)$ and set $\lambda = n\omega_1$. Then $\PQ^+(\lambda)$ is faithful for every $n \geqslant 1$ and the following descriptions hold:
\begin{align*}
	\PQm{\lambda} = &\{(n-i) \omega_1 \, : \, 0 \leqslant i \leqslant \min\{n,r\}\},\\
	&\LBQ{\lambda} = \{(n-1)\omega_1\}.
\end{align*}
Since $\min \{ \langle \mu, \alpha_1^\vee \rangle \, : \, \mu \in \PQm{r\omega_1}\} = 0$ and $\min \{ \langle \mu , \alpha_1^\vee \rangle \, : \, \mu \in \LBQ{r\omega_1}\}
= r-1$, by Corollary \ref{cor: elementi massimali 2} we get $\mathcal{H}^{\mathrm{m}}_G(\lambda) = \mathcal{H}^{\mathrm{m}}_G(r\omega_1)$ if and only if $n\geq r$, while $\mathcal{H}_G(\lambda) = \mathcal{H}_G(r\omega_1)$ for every
$n\geq 1$.
\end{ex}

\begin{ex} \label{ex: omega1 in Br}
Let $G = \mathrm{Spin}(2r+1)$ and set $\lambda = n\omega_1$. Denote $\lfloor \cdot \rfloor$ and $\lceil \cdot \rceil$ the floor and the ceiling functions. Then
$\PQ^+(\lambda)$ is faithful for every $n \geqslant \lceil r/2 \rceil$ and the following descriptions hold:
\begin{align*}
	&\PQm{\lambda} = \{\lfloor n- r/2 \rfloor \omega_1 + \omega_r, (n-1) \omega_1, n\omega_1 \},\\
	& \; \; \LBQ{\lambda} = \{\lfloor n- r/2 \rfloor \omega_1 + \omega_r, \; (n-1) \omega_1 \}.
\end{align*}
Hence by Corollary \ref{cor: elementi massimali 2} we get $\mathcal{H}^{\mathrm{m}}_G(\lambda) = \mathcal{H}^{\mathrm{m}}_G(r\omega_1)$ if and only if $\mathcal{H}_G(\lambda) = \mathcal{H}_G(r\omega_1)$ if and only if $n\geq \lceil
r/2 \rceil$.
\end{ex}

\begin{ex} \label{ex: omega2 in Ar}
Let $G = \mathrm{SL}(r+1)$ and set $\lambda = n\omega_2$. Then $\PQ^+(\lambda)$ is faithful if and only if $n \geqslant 2$ if $r$ is odd, and for every $n \geqslant 1$ if
$r$ is even. Moreover
$$\LBQ{\lambda} = \{\mu_i \, : \, 0 \leqslant i < \min\{n,r\} \},$$ where we set
$$
\mu_i = \left\{ \begin{array}{ll}
		(n-2)\omega_2 + \omega_3  & \text{ if } i = 0	\\
		(i-1) \omega_1 + (n-i) \omega_2 & \text{ if } 1 \leqslant i < \min\{n,r\}
		\end{array} \right.
$$
Since $\min \{ \langle \mu , \alpha_2^\vee \rangle \, : \, \mu \in \LBQ{r\omega_2}\} = 1$, Corollary \ref{cor: elementi massimali 2} implies that $\mathcal{H}_G(\lambda) =
\mathcal{H}_G(r\omega_2)$ if and only if $n\geq r-1$.
\end{ex}

Every little brother $\mu \in \LBQ{\lambda}$ is maximal in $\PQ^+(\lambda) \smallsetminus \{\lambda\}$ w.r.t. $\leqslant^\lambda$, however it may not be maximal w.r.t. $\leqslant$.

\begin{ex} \label{ex: omega3 in A5}
Let $G = \mathrm{SL}(6)$ and set $\lambda = n\omega_3$. Then $\PQ^+(\lambda)$ is faithful if and only if $n \geqslant 2$ and $\LBQ{\lambda} = \{\mu_1, \mu_2, \mu_3, \mu_4,
\mu_5, \mu_6, \mu_7 \}$, where we denote
\begin{itemize}
	\item[-] $\mu_1 = (n-1)\omega_3$,
	\item[-] $\mu_2 = \omega_2 + (n-2)\omega_3$,
	\item[-] $\mu_3 = (n-2)\omega_3 + \omega_4$,
	\item[-] $\mu_4 = (n-2)\omega_3 + \omega_5$,
	\item[-] $\mu_5 = \omega_1 + (n-3)\omega_3$,
	\item[-] $\mu_6 = 2\omega_2 + (n-3)\omega_3 = \mu_4 + \alpha_2$,
	\item[-] $\mu_7 = (n-3)\omega_3 + 2\omega_4 = \mu_5 + \alpha_4$.
\end{itemize}
Since $\min \{\langle \mu, \alpha_3^\vee \rangle \, : \, \mu \in \LBQ{5\omega_3}\}= 2$, Corollary \ref{cor: elementi massimali 2} shows that $\mathcal{H}_G(\lambda) =
\mathcal{H}_G(5\omega_3)$ if and only if $n \geqslant 3$.
\end{ex}

Suppose that $G$ is simply connected of type $\mathsf{C}_r$ with $r\geq 1$ (where we set $\mathsf{C}_2 = \mathsf{B}_2$ and $\mathsf{C}_1 = \mathsf{A}_1$). Then the description of the
set $\LBQ{\lambda}$ is particularly simple.

\begin{prop} \label{prop: LB Sp}
Suppose $G = \mathrm{Sp}(2r)$ with $r\geq 1$ and let $\lambda \in \Lambda^+$. Denote $q = \max\{i \leqslant r \, : \, \alpha_i \in \Supp(\lambda)\}$ and set
$$
	\lambda^\mathrm{lb}_G = \lambda - \sum_{i=q}^{r-1} \alpha_i - \frac{1}{2} \alpha_r = \lambda + \omega_{q-1} - \omega_q:
$$
then $\lambda^\mathrm{lb}_G$ is the unique little brother of $\lambda$ w.r.t. $G$.
\end{prop}

\begin{proof}
Suppose first that $\LB{\lambda} \neq \varnothing$ and denote $\lambda^\mathrm{lb}_{\mathrm{ad}}$ the adjoint little brother of $\lambda$. Then $q=r \geqslant 2$ and $\lambda -
\lambda^\mathrm{lb}_{\mathrm{ad}} = \alpha_{r-1} + \alpha_r$, hence
$$
	\lambda^\mathrm{lb}_G - \lambda^\mathrm{lb}_{\mathrm{ad}} = \alpha_{r-1} + \frac{1}{2}\alpha_r = \frac{1}{2}(2\alpha_{r-1} + \alpha_r) \in \mathbb{Q}^+[\Phi^+(\lambda)]
$$
and we get $\lambda^\mathrm{lb}_{\mathrm{ad}} <^\lambda_\mathbb{Q} \lambda^\mathrm{lb}_G$.

Suppose now that $\mu \in \PQ^+(\lambda) \smallsetminus \Pi^+(\lambda)$ and denote $\lambda - \mu = \sum_{i=1}^r a_i \alpha_i$. By the description of the fundamental weights
of $G$, it follows that $a_i \in \mathbb{N}$ for every $i < r$, while $a_r \in \frac{1}{2}\mathbb{N}$: in particular we must have $a_r = \frac{n}{2}$ for some odd
integer $n > 0$, since otherwise we would have $\mu < \lambda$. Notice that for $i >q$ we have
$$
	\langle \mu, \alpha_i^\vee \rangle =
	\left\{ \begin{array}{ll}
	a_{i-1} - 2a_i + a_{i+1} & \text{ if } i < r-1\\
	a_{r-2} - 2a_{r-1} + 2a_r & \text{ if } i = r-1\\
	a_{r-1} - 2a_r & \text{ if } i = r
	\end{array} \right.	
$$
Since $\mu$ is dominant, it follows that $a_q \geqslant a_{q+1} \geqslant \ldots \geqslant a_{r-1} \geqslant 2a_r$. Since $a_r \geqslant 1/2$, we get then $a_i \geqslant 1$ for every
$i$ such that $q \leqslant i < r$, and it follows that $\mu \leqslant \lambda^\mathrm{lb}_G$.

We now show that every $\mu \in \Pi^+(\lambda^\mathrm{lb}_G)$ satisfies the inequality $\mu \leqslant^\lambda \lambda^\mathrm{lb}_G$. Let $\mu \in \Pi^+(\lambda^\mathrm{lb}_G)$, set $\mu' = \mu -
\omega_{q-1} + \omega_q$ and denote $w(\mu')$ its $W$-conjugate in the dominant Weyl chamber. To conclude it is enough to show that $w(\mu') \in \Pi^+(\lambda)$:
indeed then $\mu' \in \Pi(\lambda)$ and by Proposition \ref{prop: lambda dominanza in natura} it follows $\mu \leqslant^\lambda \lambda^\mathrm{lb}_G$. The claim is clear if
$\mu'$ is dominant, so we may assume that $\langle \mu , \alpha_{q-1}^\vee \rangle = 0$. Define $p = \max\{i < q \, : \, \langle \mu , \alpha_i^\vee \rangle \neq
0\}$ or set $p= 0$ in case $\Supp(\mu) \cap \{\alpha_1, \ldots, \alpha_{q-1}\} = \varnothing$: then $w$ is the reflection associated to the root $\beta = \alpha_{p+1}+
\ldots + \alpha_{q-1}$ and $w(\mu') = \mu' + \beta$. On the other hand, since $\alpha_{q-1} \in \Supp(\lambda^\mathrm{lb}_G) \smallsetminus \Supp(\mu)$, by the definition of $p$
it follows that $\{\alpha_{p+1}, \ldots, \alpha_{q-1}\} \subset \Supp_S(\lambda^\mathrm{lb}_G - \mu)$.
\end{proof}

If $\Delta' \subset \Delta$ is a connected component, denote by $\Lambda_{\Delta'} \subset \Lambda$ the associated weight lattice. If $\lambda^\mathrm{lb}_{\Delta'} \in \LB{\lambda}$
is an adjoint little brother of $\lambda$ which arises correspondingly to a non-simply laced connected component $\Delta' \subset \Delta$ such that $\mathcal{X}(T) \cap
\Lambda_{\Delta'} = \mathbb{Z}[\Delta']$, then it follows from the definition that $\lambda^\mathrm{lb}_{\Delta'} \in \LBQ{\lambda}$. For instance this happens whenever $\Delta'$ is of
type $\mathsf{F} \mathsf{G}$. On the other hand, if $\Delta'$ is of type $\mathsf{C}_r$ with $r \geqslant 2$, then previous proposition shows that we may have $\lambda^\mathrm{lb}_{\Delta'}
\not \in \LBQ{\lambda}$. This never happens if $\Delta'$ is of type $\mathsf{B}_r$ with $r\geq 3$.

\begin{prop}
Suppose that $G = \mathrm{Spin}(2r+1)$ with $r \geqslant 3$, let $\lambda \in \Lambda^+$ be such that $\alpha_r \not \in \Supp(\lambda)$ and let $\lambda^\mathrm{lb}_\mathrm{ad} \in
\LB{\lambda}$ be the adjoint little brother of $\lambda$. Then $\lambda^\mathrm{lb}_\mathrm{ad} \in \LBQ{\lambda}$.
\end{prop}

\begin{proof}
Denote $p < r$ the maximum such that $\alpha_p \in \Supp(\lambda)$ and recall that $\lambda^\mathrm{lb}_\mathrm{ad} = \lambda -(\alpha_p + \ldots + \alpha_r)$. Suppose that $\mu
\in \PQ^+(\lambda)$ is such that $\lambda^\mathrm{lb}_\mathrm{ad} <^\lambda_\mathbb{Q} \mu <_\mathbb{Q} \lambda $ and denote $\lambda - \mu = \sum_{i=1}^r a_i \alpha_i$. Denote $t\geq p$ the
maximum such that $a_t \neq 0$: then we must have $t = r$, since otherwise $\langle \mu, \alpha_{t+1} \rangle < 0$. Hence we get $\Supp_S(\lambda - \mu) =
\Supp_S(\lambda - \lambda^\mathrm{lb}_\mathrm{ad})$ and in particular it follows that $\mu \in \PQ^+(\lambda) \smallsetminus \Pi^+(\lambda)$.

By the description of the fundamental weights of $G$, it follows that $a_i \in \mathbb{Z}$ for every even $i$, while $a_i \in \frac{1}{2}\mathbb{Z}$ for every odd $i$.
If $1 < i < r$, notice that
$$
	(a_{i+1} - a_i) - (a_i - a_{i-1}) = \langle \lambda - \mu, \alpha_i^\vee \rangle \in \mathbb{Z}.
$$
Since $\mu \not \in \Pi^+(\lambda)$, it follows then $a_i \not \in \mathbb{Z}$ for every odd $i$: therefore we must have $p = 1$ and $a_i = 1/2$ for every odd $i$.
On the other hand, since $\mu$ is dominant and since $p = 1$, it is easy to show that $a_1 \geqslant a_2 \geqslant \ldots \geqslant a_r$: hence we get $a_2 = 1/2$, which
is absurd.
\end{proof}

\section{Normality.}

Recall from Section 1 that $M$ is the canonical compactification of $G$ and, if $\lambda \in \mathcal{X}(T)^+$, recall the graded algebra
$$
	\widetilde A(\lambda) = \bigoplus_{n=0}^{\infty} \Gamma(M,\mathcal{M}_{n\lambda}).
$$
If $G$ is adjoint, Kannan proved in \cite{Ka} that $\widetilde A(\lambda)$ is generated in degree one, however this is false for general semisimple groups. A
counterexample is the following.

\begin{ex}	\label{ex: controesempio suriettivita}
Suppose $G = \mathrm{SL}(5)$ and set $\lambda = \omega_1 + \omega_4$, $\nu = \omega_1 + \omega_2$.
Then $\nu \leqslant_\mathbb{Q} 2 \lambda$, hence $\End(V(\nu))^* \subset \Gamma(M,\mathcal{M}_{2\lambda})$.
However they do not exist dominant weights $\mu_1 \leqslant_\mathbb{Q} \lambda$ and $\mu_2 \leqslant_\mathbb{Q} \lambda$
such that $V(\nu) \subset V(\mu_1) \otimes V(\mu_2)$. In particular, the multiplication map
\[
	\Gamma(M,\mathcal{M}_\lambda) \times \Gamma(M,\mathcal{M}_\lambda) \longrightarrow \Gamma(M,\mathcal{M}_{2\lambda})
\]
is not surjective and $\widetilde A(\lambda)$ is not generated in degree one.
\end{ex}

We now show that, if $\lambda \in \mathcal{X}(T)^+$ is sufficiently regular, then $\widetilde A(\lambda)$ is generated in degree one. Before that, we recall a lemma from \cite{Ka}.

\begin{lem} [{\cite[Lemma 3.2]{Ka}}]
Let $\lambda, \mu \in \Lambda^+$ and let $\nu \in \Pi^+(\lambda + \mu)$.
They exist $\lambda' \in \Pi^+(\lambda)$ and $\mu' \in \Pi^+(\mu)$ such that
$V(\nu) \subset V(\lambda') \otimes V(\mu')$.
\end{lem}

\begin{prop} \label{prop: generatori projcoord}
Suppose that $\lambda \in \mathcal{X}(T)^+$ is such that $\langle \lambda, \alpha^\vee \rangle \geqslant r_\alpha$ for every $\alpha \in \Supp(\lambda) \smallsetminus \Supp(\lambda)^\circ$.
Then the algebra $\widetilde A(\lambda)$ is generated in degree one.
\end{prop}

\begin{proof}
By previous corollary it is enough to show that for every $\mu \in \PQ^+(n\lambda)$ which is maximal w.r.t. $\leqslant$ the submodule $\End(V(\mu))^* \subset \widetilde
A_n(\lambda)$ is contained in the power $\widetilde A_1(\lambda)^n$. Since it is maximal w.r.t. $\leqslant$, the weight $\mu$ is maximal also w.r.t. $\leqslant^\lambda$, hence by
Proposition \ref{prop: elementi massimali} and Corollary \ref{cor: elementi massimali} we get that
$$
	\langle n\lambda - \mu, \alpha^\vee \rangle \leqslant \left\{
	\begin{array}{ll}
	1 & \text{ if } \alpha \in \Supp(\lambda)^\circ \\
	r_\alpha & \text{ if } \alpha \in \Supp(\lambda) \smallsetminus \Supp(\lambda)^\circ
	\end{array} \right.	
$$
and by the assumption on the coefficients of $\lambda$ it follows that $\langle n\lambda - \mu, \alpha^\vee \rangle \leqslant \langle \lambda, \alpha^\vee \rangle$.
Therefore, if we set $\mu' = \mu - (n-1)\lambda$, then we get $\mu ' \in \PQ^+(\lambda)$ and $V(\mu) \subset V(\mu') \otimes V(\lambda)^{\otimes n-1}$, hence
$\End(V(\mu))^* \subset \widetilde A_1(\lambda)^n$.
\end{proof}

Previous proposition allows to describe the normalization of a simple linear compactification of $G$ as follows. If $\lambda \in \mathcal{X}(T)^+$ is an almost
faithful weight, consider the adjoint compactification $X_\lambda$ and denote $\widetilde X_\lambda \longrightarrow X_\lambda$ the normalization of $X_\lambda$ in $\Bbbk(G)$. Before
describing $\widetilde X_\lambda$ we recall a lemma from \cite{CDM}, which is given there in the context of symmetric varieties.

\begin{lem} [{\cite[Prop. 2.1]{CDM}}] \label{lem: anelli integrali}
The algebra $\widetilde A(\lambda)$ is integral over $A(\lambda)$.
\end{lem}

\begin{teo} \label{teo: descrizione normalizzazione}
Let $\lambda \in \mathcal{X}(T)^+$ be an almost faithful weight with $\langle \lambda, \alpha^\vee \rangle \geqslant r_\alpha$ for every $\alpha \in \Supp(\lambda) \smallsetminus
\Supp(\lambda)^\circ$. Then $\widetilde X_\lambda \simeq X_{\PQ^+(\lambda)}$.
\end{teo}

\begin{proof}
Since $M$ is a normal variety, $\widetilde{A}(\lambda)$ is an integrally closed algebra. On the other hand, Proposition \ref{prop: generatori projcoord} shows that
$\widetilde{A}(\lambda)$ is generated in degree one by
$$
	\Gamma(M,\mathcal{M}_\lambda) \simeq \bigoplus_{\mu \in \PQ^+(\lambda)}\End(V(\mu)),
$$
so that $X_{\PQ^+(\lambda)} \subset \mathbb{P}\big(\Gamma(M,\mathcal{M}_\lambda)^*\big)$ is a projectively normal variety.

Since $\lambda$ is almost faithful it follows that $\mathcal{X}(T)^+ = \bigcup_{n \in \mathbb{N}} \PQ^+(n\lambda)$, hence there exists $n \in \mathbb{N}$ such that $\PQ^+(n\lambda)$ is
faithful. On the other hand $\PQ^+(n\lambda)$ is faithful if and only if the set
$$
\mathcal{H}^{\mathrm{m}}_G(n\lambda) = \{\mu - n\lambda \, : \, \mu \in \PQ^+(n\lambda) \text{ is maximal w.r.t. } \leqslant^\lambda \}
$$
generates $\mathcal{X}(T)/\mathbb{Z}[\Delta]$: therefore $\PQ^+(\lambda)$ as well is faithful, since Corollary \ref{cor: elementi massimali 2} implies $\mathcal{H}^{\mathrm{m}}_G(\lambda) =
\mathcal{H}^{\mathrm{m}}_G(n\lambda)$ for every $n>0$.

Consider now the subalgebra $A(\lambda) \subset \widetilde A(\lambda)$ generated by
$\End(V(\lambda))^* \subset \widetilde A_1(\lambda)$, which is identified with the homogeneous coordinate ring of $X_\lambda$. Since $\widetilde{A}(\lambda)$ is integral over
$A(\lambda)$ by Lemma \ref{lem: anelli integrali}, it follows then that $X_{\PQ^+(\lambda)}$ is the normalization of $X_\lambda$ in $\Bbbk(G)$.
\end{proof}

As shown by following examples, the hypothesis on the coefficients of $\lambda$ in previous theorem is necessary.

\begin{ex}
\begin{itemize}
	\item[i)] Set $G = \mathrm{SL}(r+1)$ and $\Pi = \PQ^+(n\omega_1)$. Then $\Pi$ is faithful for every $n \geqslant 1$ and $X_\Pi$ is normal for every $n \geqslant
1$ (see Example \ref{ex: omega1 in Ar} and Theorem \ref{teo: caratterizzazione normalita}).
	\item[ii)] Set $G = \mathrm{Spin}(2r+1)$ and $\Pi = \PQ^+(n\omega_1)$. Then $\Pi$ is faithful for every $n \geqslant \lceil r/2 \rceil$ and $X_\Pi$ is
normal for every $n \geqslant \lceil r/2 \rceil$ (see Example \ref{ex: omega1 in Br} and Theorem \ref{teo: caratterizzazione normalita}).
	\item[iii)] Set $G = \mathrm{SL}(r+1)$ and $\Pi = \PQ^+(n\omega_2)$. Then $\Pi$ is faithful if and only if $n \geqslant 2$ if $r$ is odd and for every
$n\geq 1$ if $r$ is even, whereas $X_\Pi$ is normal if and only $n \geqslant r-1$ (see Example \ref{ex: omega2 in Ar} and Theorem \ref{teo: caratterizzazione
normalita}).
	\item[iv)] Suppose that $G$ is a direct product of adjoint groups and of groups of type $\mathsf{C}_r$ with $r\geq 1$. Then $X_{\PQ^+(\lambda)}$ is a normal
compactification of $G$ for every almost faithful $\lambda \in \mathcal{X}(T)^+$ (see Proposition \ref{prop: LB Sp} and Theorem \ref{teo: caratterizzazione
normalita}).
\end{itemize}
\end{ex}

Combining previous theorem with Proposition \ref{prop: criterio tensoriale morfismi}
we get the following tensorial characterization of the normality of $X_\Pi$.

\begin{prop} \label{prop: criterio tensoriale normalita}
Let $\Pi \subset \mathcal{X}(T)^+$ be simple and faithful with maximal element $\lambda$ and suppose
that $\langle \lambda, \alpha^\vee \rangle \geqslant r_\alpha$ for every $\alpha \in \Supp(\lambda) \smallsetminus \Supp(\lambda)^\circ$.
Then the variety $X_\Pi$ is normal if and only if for every $\nu \in \PQ^+(\lambda)$
they exist $n \in \mathbb{N}$ and $\mu_1,\ldots,\mu_n \in \Pi$ such that
\[ V(\nu + (n-1)\lambda) \subset V(\mu_1) \otimes \cdots \otimes V(\mu_n). \]
\end{prop}

We may restate previous criterion in a more combinatorial way as follows.
If $\Pi\subset \mathcal{X}(T)^+$ is simple with maximal element $\lambda$, denote
$$
	\Omega(\Pi) = \Big\{ \nu - n\lambda \, : \, V(\nu) \subset \Big(\bigoplus_{\mu \in \Pi} V(\mu)\Big)^{\otimes n} \Big\}:
$$
it is a semigroup and by Lemma \ref{lem: coefficientimatriciali} it is the image of
$$
	\Quot{\Bbbk[X^\circ_\Pi]^{(B\times B^-)}}{\Bbbk^*} \subset \mathcal{X}(T) \times \mathcal{X}(T)
$$
in $\mathcal{X}(T)$ via the projection on the first factor.
If $\Pi = \{\mu_1, \ldots, \mu_m\}$, for simplicity sometimes we will denote
$\Omega(\Pi)$ also by $\Omega(\mu_1, \ldots, \mu_m)$.
Since $X_\Pi$ possesses an open $B\times B^-$-orbit,
every $B\times B^-$-semi-invariant function $\phi \in \Bbbk(X_\Pi)^{(B\times B^-)}$ is uniquely
determined by its weight up to a non-zero scalar factor, therefore we may restate
Proposition \ref{prop: criterio tensoriale morfismi}
and Proposition \ref{prop: criterio tensoriale normalita} as follows.

\begin{prop} \label{prop: criteri semigruppi}
Let $\Pi\subset \mathcal{X}(T)^+$ be simple and faithful with maximal element $\lambda$.
\begin{itemize}
	\item[i)]  Let $\Pi'\subset \mathcal{X}(T)^+$ be simple with maximal element $\lambda'$
and suppose that $\Supp(\lambda) = \Supp(\lambda')$. There exists an equivariant morphism
$X_\Pi \longrightarrow X_{\Pi'}$ if and only if $\Omega(\Pi') \subset \Omega(\Pi)$ if and only if
$\mu' - \lambda' \in \Omega(\Pi)$ for every $\mu' \in \Pi'$.
\item[ii)] Suppose that $\langle \lambda, \alpha^\vee \rangle \geqslant r_\alpha$
for every $\alpha \in \Supp(\lambda) \smallsetminus \Supp(\lambda)^\circ$.
Then the variety $X_\Pi$ is normal if and only if $\nu - \lambda \in \Omega(\Pi)$
for every $\nu \in \PQ^+(\lambda)$.
\end{itemize}
\end{prop}

\begin{oss}
Let $\lambda\in \mathcal{X}(T)^+$ be an almost faithful weight such that $\langle \lambda, \alpha^\vee \rangle \geqslant r_\alpha$ for every $\alpha \in \Supp(\lambda) \smallsetminus
\Supp(\lambda)^\circ$ and denote $\widetilde\Omega(\lambda)$ the image of the semigroup
$$
	\{B \times B^- \text{-weights in } \Bbbk[\widetilde X^\circ_\lambda] \} \subset \mathcal{X}(T) \times \mathcal{X}(T)
$$
in $\mathcal{X}(T)$ via the projection on the first factor: then $\widetilde\Omega(\lambda) = \{\pi \in \mathcal{X}(T) \, : \, n\pi \in \Omega(\lambda) \; \exists\, n \in \mathbb{N}\}$ is the
\textit{saturation} of $\Omega(\lambda)$ in $\mathcal{X}(T)$. By the description of $\widetilde X_\lambda$ given in Theorem \ref{teo: descrizione normalizzazione}, we may
describe $\widetilde\Omega(\lambda)$ more explicitly as follows:
$$
	\widetilde\Omega(\lambda)  = \left\{\mu - n\lambda \, : \, \mu \in \PQ^+(n\lambda)\right\}.
$$
Consider now the cone $\mathcal{C}(\lambda) \subset \mathcal{X}(T)_\mathbb{Q}$ generated by $\Phi^-(\lambda)$.
Then Proposition \ref{prop: lambda coni} shows that
$$
	\mathcal{C}(\lambda) \cap \mathcal{X}(T) = \left\{\mu - n\lambda \, : \, \mu \in \PQ(n\lambda)\right\}
$$
and by Theorem \ref{teo: descrizione normalizzazione} it follows that $\mathcal{C}(\lambda) \cap \mathcal{X}(T)$
is the image of
$$
	\{T \times T \text{-weights in } \Bbbk[\widetilde X^\circ_\lambda] \} \subset \mathcal{X}(T) \times \mathcal{X}(T)
$$
in $\mathcal{X}(T)$ via the projection on the first factor.
\end{oss}

Thanks to the description of $\widetilde X_\lambda$ given
in Theorem \ref{teo: descrizione normalizzazione},
we are now able to characterize the normality of a simple
linear compactification of $G$ as follows.

\begin{teo} \label{teo: caratterizzazione normalita}
Let $\Pi\subset \mathcal{X}(T)^+$ be a simple subset with maximal element $\lambda$ and assume that $\langle \lambda, \alpha^\vee \rangle \geqslant r_\alpha$ for every $\alpha
\in \Supp(\lambda) \smallsetminus \Supp(\lambda)^\circ$. Then the variety $X_\Pi$ is a normal compactification of $G$ if and only if $\Pi \supset \LBQ{\lambda}$.
\end{teo}

Notice that the hypothesis on the coefficients of $\lambda$ in the previous theorem is automatically fulfilled whenever $\lambda$ is a regular weight. Notice
also that it involves no loss of generality: if $\alpha \in \Supp(\lambda)$ is such that $\langle \lambda, \alpha^\vee \rangle < r_\alpha$ then we may consider the
simple set
$$
	\Pi' = \{\mu + (r_\alpha - \langle \lambda, \alpha^\vee \rangle ) \omega_\alpha \, : \, \mu \in \Pi\}
$$
and by Proposition \ref{prop: criterio tensoriale morfismi} the varieties $X_\Pi$ and $X_{\Pi'}$ are equivariantly isomorphic. However, it may also happen
that the hypothesis on the coefficients of $\lambda$ is not needed: for instance this is the case if $G = \mathrm{SL}(r+1)$ and $\Supp(\lambda)=\{\alpha_1\}$ (see
Example \ref{ex: omega1 in Ar}) and when $G$ is a direct product of adjoint groups and groups of type $\mathsf{C}_r$ with $r \geqslant 1$ (see Proposition \ref{prop:
LB Sp}).

For convenience we split the proof of previous theorem in the following two propositions, where we treat separately the necessity and the sufficiency of
the condition.

\begin{prop}
Let $\Pi\subset \mathcal{X}(T)^+$ be simple and faithful with maximal element $\lambda$ and assume that $\langle \lambda, \alpha^\vee \rangle \geqslant r_\alpha$ for every
$\alpha \in \Supp(\lambda) \smallsetminus \Supp(\lambda)^\circ$. If $X_\Pi$ is normal, then $\Pi \supset \LBQ{\lambda}$.
\end{prop}

\begin{proof}
By Proposition \ref{prop: criterio tensoriale normalita} for every $\nu \in \PQ^+(\lambda)$ they exist $n \in \mathbb{N}$ and $\mu_1,\ldots,\mu_n \in \Pi$ such that
$$
	V(\nu + (n-1)\lambda) \subset V(\mu_1) \otimes \cdots \otimes V(\mu_n).
$$
Suppose that $\nu \in \PQ^+(\lambda) \smallsetminus \Pi^+(\lambda)$ is maximal w.r.t. $\leqslant^\lambda_\mathbb{Q}$. Since $\sum_{i=1}^r \mu_i - \nu - (n-1)\lambda \in \mathbb{N}[\Delta]$, we
must have $\mu_i \in \PQ^+(\lambda) \smallsetminus \Pi^+(\lambda)$ for some $i$: hence by Proposition \ref{prop: necessita} we get $\nu \leqslant^\lambda_\mathbb{Q} \mu_i$ and the
equality $\mu_i = \nu$ follows by maximality. Suppose now that $\nu \in \LB{\lambda}$ is maximal in $\PQ^+(\lambda) \smallsetminus \{\lambda\}$ w.r.t.$\leqslant_\mathbb{Q}^\lambda$: then
by Proposition \ref{prop: necessita} for every $i$ we must have either $\mu_i = \lambda$ or $\mu_i = \nu$, while by Theorem \ref{teo: normalita caso aggiunto}
we cannot have $\mu_i = \lambda$ for every $i$.
\end{proof}

Before to prove the sufficiency of the condition, we recall the former Parthasarathy-Ranga Rao-Varadarajan conjecture, which was proved independently by
Kumar \cite{Ku1} and Mathieu \cite{Ma}.

\begin{teo}[PRV Conjecture]
Let $\lambda, \mu \in \mathcal{X}(T)^+$ be dominant weights and let $\nu \leqslant \lambda + \mu$
be a dominant weight of the shape $\nu = w\lambda + w'\mu$, with $w,w'\in W$.
Then $V(\nu) \subset V(\lambda) \otimes V(\mu)$.
\end{teo}

\begin{prop}
Let $\Pi \subset \mathcal{X}(T)^+$ be a simple subset whose maximal element $\lambda$ is such that $\langle \lambda, \alpha^\vee \rangle \geqslant r_\alpha$ for every $\alpha \in
\Supp(\lambda) \smallsetminus \Supp(\lambda)^\circ$. If $\Pi \supset \LBQ{\lambda}$, then $X_\Pi$ is a normal compactification of $G$.
\end{prop}

\begin{proof}
Suppose that $X_{\LBQb{\lambda}}$ is a normal compactification of $G$:
then the claim follows by considering the natural projections
$$
	X_{\PQ^+(\lambda)} \longrightarrow X_\Pi \longrightarrow X_{\LBQb{\lambda}}.
$$
Therefore we are reduced to the case $\Pi = \LBQb{\lambda}$.
Since by Theorem \ref{teo: descrizione normalizzazione} $X_{\PQ^+(\lambda)}$ is a normal compactification of $G$, by Proposition \ref{prop: criterio
tensoriale normalita},
it is enough to show that $\nu-\lambda \in \Omega\big(\LBQb{\lambda}\big)$
for every $\nu \in \PQ^+(\lambda)$. Denote
$$
	\PQo{\lambda} = (\PQ^+(\lambda) \smallsetminus \Pi^+(\lambda)) \cup \LB{\lambda}:
$$
by the normality of the adjoint compactification $X_{\LBb{\lambda}}$ (see Theorem \ref{teo: normalita caso aggiunto}), it is sufficient to show that $\nu -
\lambda \in \Omega(\LBQb{\lambda})$ for any $\nu \in \PQo{\lambda}$.

Suppose that $\nu \in \PQo{\lambda}$ is not maximal w.r.t. $\leqslant^\lambda_\mathbb{Q}$ and let $\mu_1 \in \LBQ{\lambda}$ be such that $\nu <_\mathbb{Q}^\lambda \mu_1$: then $\nu -
\mu_1 \in \mathcal{C}(\lambda)$, hence by Proposition \ref{prop: lambda coni} there exists $m \in \mathbb{N}$ such that $m\lambda + \nu - \mu_1 \in \mathcal{P}(m\lambda)$. Denote
$$
	\nu' = \nu + (m -1)\lambda, \qquad \mu_1' = \mu_1 + (m -1)\lambda,
	\qquad \lambda' = m\lambda.
$$
and denote $\nu_2 = w(\lambda' + \nu' - \mu_1') \in \mathcal{X}(T)^+$ the dominant weight which is conjugated to $\lambda' + \nu' - \mu_1'$ under the Weyl group $W$.
Then $\nu' \leqslant^\lambda_\mathbb{Q} \mu_1' <_\mathbb{Q} \lambda'$ and, since $\mathcal{P}(\lambda')$ is $W$-stable, it follows that $\nu_2 \leqslant_\mathbb{Q} \lambda'$. Moreover $\nu' + \lambda' =
\mu_1' + w^{-1}\nu_2$, hence the PRV conjecture implies
$$
	V(\nu' + \lambda') \subset V(\mu_1') \otimes V(\nu_2).
$$
Suppose that $\nu_2 \in \LBQ{\lambda'}$: then by Corollary \ref{cor: elementi massimali 2} there exists $\mu_2 \in \LBQ{\lambda}$ such that $\nu_2 = \mu_2 +
(m-1) \lambda$. Hence $\lambda - \mu_1 = \lambda' - \mu_1'$ and $\lambda - \mu_2 = \lambda' - \nu_2$ and by Proposition \ref{prop: criteri semigruppi} i) we get
$$
	\nu - \lambda = \nu' - \lambda' \in \Omega(\lambda', \mu_1', \nu_2) =
	\Omega(\lambda, \mu_1, \mu_2) \subset \Omega\big(\LBQb{\lambda}\big).
$$

Suppose now that $\nu_2 \in \PQo{\lambda'}$ is not maximal w.r.t. $\leqslant^\lambda_\mathbb{Q}$. Since $\nu' + \lambda' - \mu_1' \leqslant \nu_2$, it follows that
$$
	\lambda' - \nu_2 \leqslant \lambda' - (\lambda' + \nu' - \mu'_1) = \mu_1 - \nu <_\mathbb{Q} \lambda - \nu:
$$
hence we may proceed by decreasing induction on the rational dominance order and we may assume that they exist $n > 2$ and $\mu_2', \ldots , \mu_n' \in
\LBQ{\lambda'}$ such that
$$
	V(\nu_2 + (n-2) \lambda') \subset V(\mu'_2) \otimes \ldots \otimes V(\mu'_n).
$$
Combining previous tensorial inclusions and applying Corollary \ref{cor: traslazione} we get then
$$
	V(\nu' + (n-1) \lambda') \subset V(\mu_1') \otimes V(\mu'_2) \otimes \ldots \otimes V(\mu'_n).
$$
By Corollary \ref{cor: elementi massimali 2} they exist $\mu_2, \ldots , \mu_n \in \LBQ{\lambda}$ such that $\mu_i' = \mu_i + (m-1) \lambda$ for every $i = 2,
\ldots, n$. Hence $\lambda - \mu_i = \lambda' - \mu_i'$ for every $i$ and by Proposition \ref{prop: criteri semigruppi} i) we get
\[
	\nu - \lambda = \nu' - \lambda' \in \Omega(\lambda', \mu_1', \ldots, \mu_n') =
	\Omega(\lambda, \mu_1, \ldots, \mu_n) \subset \Omega \big( \LBQb{\lambda} \big). \qedhere
\]
\end{proof}

\section{Smoothness}

If $\Pi \subset \mathcal{X}(T)^+$ is simple and faithful with maximal element $\lambda$, then the normalization of $X_\Pi$ coincides with $\widetilde X_\lambda$: hence by
Theorem \ref{teo: caratterizzazione normalita} it follows that $X_\Pi$ is smooth if and only if $\widetilde X_\lambda$ is smooth and $\Pi \supset \LBQ{\lambda}$.
Therefore the problem of studying the smoothness of $X_\Pi$ reduces to the study of the smoothness of $\widetilde X_\lambda$.

We now recall some results about the $\mathbb{Q}$-factoriality and the smoothness of $\widetilde X_\lambda$. Recall that a normal variety $X$ is called \textit{locally
factorial} if $\Pic(X) = \mathrm{Cl}(X)$, while it is called $\mathbb{Q}$-\textit{factorial} if $\Pic(X)_\mathbb{Q} = \mathrm{Cl}(X)_\mathbb{Q}$. In the case of a simple
spherical variety, these properties are nicely expressed by the combinatorial properties of the colored cone.

\begin{prop} [{\cite[Rem. 2.2.ii]{Br1}, \cite[Prop. 4.2]{Br2}}]	\label{prop: fattorialita}
Let $\lambda \in \mathcal{X}(T)^+$.
\begin{itemize}
	\item[i)] $\widetilde X_\lambda$ is $\mathbb{Q}$-factorial if and only if $\mathcal{C}(\widetilde X_\lambda)$
	is a simplicial cone (i.e. generated by linearly independent vectors).
	\item[ii)] $\widetilde X_\lambda$ is locally factorial if and only if $\mathcal{C}(\widetilde X_\lambda)$
	is generated by a basis of $\mathcal{X}(T)^\vee$.
\end{itemize}
\end{prop}

More explicitly, the $\mathbb{Q}$-factoriality of $\widetilde X_\lambda$ can be characterized as follows, where we denote by $\Delta^e$ the set of the extremal roots of
$\Delta$.

\begin{prop}[{\cite[Prop.~3.4 and Cor.~3.5]{BGMR}}] \label{prop: Q-fattorialita}
Let $\lambda \in \mathcal{X}(T)^+$. The variety $\widetilde X_\lambda$ is $\mathbb{Q}$-factorial
if and only if the following conditions are fulfilled:
\begin{itemize}
	\item[i)] For every connected component $\Delta'$ of $\Delta$,
$\Supp(\lambda)\cap \Delta'$ is connected and, in case it contains a unique element,
then this element is an extremal root of $\Delta'$;
	\item[ii)] $\Supp(\lambda)$ contains every simple root which is adjacent
	to three other simple roots and at least two of the latter.
\end{itemize}
If this is the case, then the extremal rays of $\mathcal{C}(\widetilde X_\lambda)$ are the half-lines generated by
the elements in the set
$$
\{	\alpha^{\vee} \, : \, \alpha\in \Delta \smallsetminus \Supp(\lambda) \} \cup
\{ -\omega_{\alpha}^{\vee} \, : \,  \alpha\in \Supp(\lambda)^\circ \cup \big(\Delta^e \smallsetminus \Supp(\lambda)\big)\}.
$$
\end{prop}

Previous result is stated in \cite{BGMR} in the case of a semisimple adjoint group, however it holds for any semisimple group since the cone $\mathcal{C}(\widetilde
X_\lambda)$ depends only on the set $\Supp(\lambda)$. Previous proposition allows to reduce the study of the smoothness of the variety $\widetilde X_\lambda$ to the case
of a simple group $G$ as follows.

\begin{lem}	\label{lem: riduzione G semplice}
If $\lambda \in \mathcal{X}(T)^+$ is an almost faithful weight such that $\widetilde X_\lambda$ is locally factorial,
then $G = G_1 \times \ldots \times G_n$ is a direct product of simple groups.
If moreover $T_i = T\cap G_i$ and $\lambda_i = \lambda \bigr|_{T_i}$,
then
$$
	\widetilde X_\lambda = \widetilde X_{\lambda_1} \times \ldots \times \widetilde X_{\lambda_n},
$$
where $\widetilde X_{\lambda_i}$ denotes the normalization of $X_{\lambda_i}$ in $\Bbbk(G_i)$.
\end{lem}

\begin{proof}
Denote $G_1, \ldots, G_n$ the simple factors of $G$ and denote $\Lambda^i_{\mathrm{rad}}$ the root lattice of $G_i$: then $\Lambda_{\mathrm{rad}} =
\Lambda^1_{\mathrm{rad}} \times \ldots \times \Lambda^n_{\mathrm{rad}}$ and $G$ is the direct product of its simple factors if and only if $\mathcal{X}(T)^\vee$ is the
direct product of the intersections $\mathcal{X}(T)^\vee\cap (\Lambda^i_{\mathrm{rad}})^\vee_\mathbb{Q}$. Hence the first claim follows by Proposition \ref{prop:
fattorialita} ii) together with the description of the extremal rays of $\mathcal{C}(\widetilde X_\lambda)$ given in Proposition \ref{prop: Q-fattorialita}. The second
claim follows straightforwardly from the isomorphism $X_\lambda \simeq X_{\lambda_1} \times \ldots \times X_{\lambda_n}$ (see \cite[Lemma~1.1]{BGMR}).
\end{proof}

In the case of an adjoint group, the smoothness of $\widetilde X_\lambda$ has been characterized in \cite{BGMR} as follows.

\begin{teo} [{\cite[Thm.~3.6]{BGMR}}]	\label{teo: smoothness caso aggiunto}
Suppose that $G$ is simple and adjoint and let $\lambda \in \mathcal{X}(T)^+$. Then $\widetilde X_\lambda$ is smooth if and only if $X_\lambda$ is smooth if and only if $\lambda$
satisfies the following conditions:
\begin{itemize}
	\item[i)] If $\Supp(\lambda)$ contains a long root, then it contains also the unique short simple root
	which is non-orthogonal to a long simple root;
	\item[ii)] $\Supp(\lambda)$ is connected and, in case it contains a unique element, then this element is an extremal root of $\Delta$;
	\item[iii)] $\Supp(\lambda)$ contains every simple root which is adjacent to three other simple roots and at least two of the latter;
	\item[iv)] Every connected component of $\Delta\smallsetminus \Supp(\lambda)$ is of type $\mathsf{A}$.
\end{itemize}
\end{teo}

Therefore we only need to consider the case of a simple non-adjoint group. A general criterion for the smoothness of a group compactification was given by
Timashev in \cite{Ti}; for convenience, we will use a generalization which can be found in \cite{Ru} in the more general context of symmetric spaces. We
recall it in the case of the compactification $\widetilde X_\lambda$.

\begin{teo}[see {\cite[Thm.~2.2]{Ru}, \cite[Thm.~9]{Ti}}] \label{teo: smooth-timashev}
Let $\lambda \in \mathcal{X}(T)^+$. Then $\widetilde X_\lambda$ is smooth if and only if the following
conditions are fulfilled:
\begin{itemize}
\item[i)] All connected components of $\Delta \smallsetminus \Supp(\lambda)$ are of type $\mathsf{A}$ and there
  are no more than $|\Supp(\lambda)|$ of them;
\item[ii)]  The cone $\mathcal{C}(\widetilde X_\lambda)$ is simplicial and it is generated by a
  basis $\mathcal{B}(\lambda) \subset \mathcal{X}(T)^\vee$;
\item[iii)]  One can enumerate the simple roots in order of their positions
  at Dynkin diagrams of connected components
  $I_k = \{\alpha_{1}^{k},\ldots,\alpha_{n_{k}}^{k}\}$ of $\Delta \smallsetminus \Supp(\lambda)$ ,
  $k=1,\ldots,n$, and partition the dual basis $\mathcal{B}(\lambda)^* \subset \mathcal{X}(T)$ into subsets
  $J_k = \{\pi_{1}^{k},\ldots,\pi_{n_{k}+1}^{k}\}$,
  $k=1,\ldots,p$ (with $p\geq n$) in such a way that
  $\langle\pi_{j}^{k}, (\alpha_{i}^{h})^\vee \rangle = \delta_{i,j}\delta_{h,k}$
  and $\pi_{j}^{k} - \frac{j}{n_{k}+1} \pi_{n_{k}+1}^{k}$
  is the $j$-th fundamental weight of the root system generated by
  $I_k$ for all $j,k$.
\end{itemize}
\end{teo}

Before applying previous theorem in the case of our interest, we need an auxiliary lemma.

\begin{lem}	\label{lem: copesi fondamentali}
Let $\lambda \in \mathcal{X}(T)^+$ be an almost faithful weight such that $\widetilde X_\lambda$ is smooth and denote $\mathcal{B}(\lambda) \subset \mathcal{X}(T)^\vee \cap \mathcal{C}(\widetilde
X_\lambda)$ the associated basis of $\mathcal{X}(T)^\vee$. If $\beta \in \Delta^e \smallsetminus \Supp(\lambda)$, then $-\omega_\beta^\vee \in \mathcal{B}(\lambda)$.
\end{lem}

\begin{proof}
Denote $\mathcal{B}(\lambda) = \{v_1, \ldots, v_r\}$ and denote $\mathcal{B}(\lambda)^* = \{v_1^*, \ldots, v_r^*\}$ the basis of $\mathcal{X}(T)$ which is dual to $\mathcal{B}(\lambda)$,
defined by $\langle v_i^*, v_j \rangle = \delta_{ij}$. Denote $\Delta \smallsetminus \Supp(\lambda) = \bigcup_{k=1}^n I_k$ the decomposition in connected components and
denote $\mathcal{B}(\lambda)^* = \bigcup_{k=1}^p J_k$ the partition which exists by Theorem \ref{teo: smooth-timashev} iii), where $p\geq n$.

Fix a connected component $I_k = \{\alpha_1^k,\ldots,\alpha_{n_k}^k\} \subset \Delta \smallsetminus \Supp(\lambda)$ and consider the associated component $J_k =
\{\pi_1^k,\ldots,\pi_{n_k+1}^k\} \subset \mathcal{B}(\lambda)^*$, where the numberings in $I_k$ and in $J_k$ are those of Theorem \ref{teo: smooth-timashev} iii).
Denote $L \subset G$ the Levi subgroup associated to $I_k$ and denote by $\omega^L_\alpha$ the fundamental weight of $L$ associated to a simple root $\alpha \in
I_k$.

By Proposition \ref{prop: Q-fattorialita} it follows that $I_k$ is of type $\mathsf{A}$ and it contains exactly an extremal root $\beta \in \Delta^e$. By the
description of the extremal rays of $\mathcal{C}(\widetilde X_\lambda)$ it follows that there is a unique element in $\mathcal{B}(\lambda) \smallsetminus I_k^\vee$ which is non-orthogonal
to $I_k$, namely the element $v_\beta \in \mathcal{B}(\lambda)$ which is proportional to $\omega_\beta^\vee$. Therefore up to renumbering we must have $J_k = \{
w_{\alpha_1^k}^*, \ldots, w_{\alpha_{n_k}^k}^*, v_\beta^* \}$, where if $\alpha \in \Delta \smallsetminus \Supp(\lambda)$ we denote by $w_\alpha \in \mathcal{B}(\lambda)$ the unique
element which is proportional with $\alpha^\vee$.

Notice that $v_\beta^* = \pi^k_{n_k +1}$: indeed $\beta \in \Supp_\Delta(\omega^L_\alpha)$ for every $\alpha \in \Delta \smallsetminus \Supp(\lambda)$, hence
$$
	\langle \pi_i^k, v_\beta \rangle - \frac{i}{n_k +1} \langle \pi_{n_k +1}^k, v_\beta \rangle = \langle \omega^L_{\alpha_i^k}, v_\beta \rangle \neq 0.
$$
Set now $v_\beta = -m\omega_\beta^\vee$, where $m > 0$, and set $\beta = \alpha_i^k$. Then by Theorem \ref{teo: smooth-timashev} iii) it follows that $\omega_\beta^L
= \pi_i^k - \frac{i}{n_k+1} \pi_{n_k+1}^k$, hence
$$
	\langle \omega^L_\beta , v_\beta \rangle = -\frac{i}{n_k+1} \langle v_\beta^*, v_\beta \rangle = -\frac{i}{n_k+1}
$$
for every $i \leqslant n_k$. On the other hand $L$ is of type $\mathsf{A}_{n_k}$, hence
$$
	\omega^L_\beta = \frac{1}{n_k+1} \left(\sum_{j=1}^{i-1} j(n_k-i+1)\alpha^k_j + \sum_{j=i}^{n_k} i(n_k - j +1) \alpha^k_j\right)
$$	
and we get
$$
	\langle \omega^L_\beta , v_\beta \rangle = -\frac{im(n_k-i+1)}{n_k+1} \langle \beta , \omega_\beta^\vee \rangle = - \frac{im(n_k - i +1)}{n_k+1}.
$$
Comparing the two equalities, it follows then $i = n_k$ and $m=1$: therefore the numbering of $I_k$ starts from the extremal root of $I_k$ which is not
extremal in $\Delta$ and $-\omega_\beta^\vee \in \mathcal{B}(\lambda)$.
\end{proof}

\begin{teo}	\label{teo: smoothness caso non aggiunto}
Suppose that $G$ is simple and non-adjoint and let
$\lambda \in \mathcal{X}(T)^+$ be an almost faithful weight.
Then $\widetilde X_\lambda$ is smooth if and only if
$G = \mathrm{Sp}(2r)$ with $r \geqslant 1$,
$\Supp(\lambda)$ is connected and $\alpha_r \in \Supp(\lambda)$.
\end{teo}

Notice that $\mathrm{Sp}(4)$ has type $\mathsf{C}_2 = \mathsf{B}_2$, while $\mathrm{Sp}(2)$ has type $\mathsf{C}_1 = \mathsf{A}_1$.

\begin{proof}
Suppose that $G = \mathrm{Sp}(2r)$ and assume that $\Supp(\lambda)$ is connected and $\alpha_r \in \Supp(\lambda)$: then Proposition \ref{prop: Q-fattorialita} and
Theorem \ref{teo: smooth-timashev} show that the variety $\widetilde X_\lambda$ is smooth.

Suppose conversely that $\widetilde X_\lambda$ is smooth. If $\lambda$ is regular, then $\mathcal{C}(\widetilde X_\lambda)$ is the negative Weyl chamber $C^-$, so it is simplicial and
condition iii) of Theorem \ref{teo: smooth-timashev} is empty. Hence we only need to verify that $C^-$ is generated by a basis of $\mathcal{X}(T)^\vee$, and it
is easy to verify that this happens if and only if $G$ is either adjoint or of type $\mathsf{C}_r$ with $r\geq 1$.

Therefore we may assume that $\lambda$ is non-regular; by Proposition \ref{prop: Q-fattorialita} it follows that $\Supp(\lambda)$ is connected and that there
exists a (unique) basis $\mathcal{B}(\lambda) \subset \mathcal{X}(T)^\vee$ which generates $\mathcal{C}(\widetilde X_\lambda)$ as a cone. Since $\Supp(\lambda)$ is connected and since
$\lambda$ is not regular, we must have $\Delta^e \smallsetminus \Supp(\lambda) \neq \varnothing$. Since otherwise $G$ is necessarily adjoint, $\Delta$ cannot be of type $\mathsf{E}_8$,
$\mathsf{F}_4$, $\mathsf{G}_2$.

\textit{Type $\mathsf{A}$.} If $\Delta$ is of type $\mathsf{A}_r$ with $r > 1$, then $\Lambda^\vee/\mathbb{Z}[\Delta^\vee]$ is generated both by $\omega_1^\vee$ and by $\omega_r^\vee$.
Since $\Delta^e \smallsetminus \Supp(\lambda) \neq \varnothing$, Lemma \ref{lem: copesi fondamentali} implies then $\mathcal{X}(T)^\vee = \Lambda^\vee$, i.e. $G$ is adjoint.

\textit{Type $\mathsf{B}$.} If $\Delta$ is of type $\mathsf{B}_r$ with $r > 2$, then $\Phi^\vee$ has type $\mathsf{C}_r$  and $\Lambda^\vee/\mathbb{Z}[\Delta^\vee]$ is generated by any
fundamental coweight $\omega_i^\vee$ with $i$ odd. If $\alpha_1 \not \in \Supp(\lambda)$, Lemma \ref{lem: copesi fondamentali} implies that $\mathcal{X}(T)^\vee =
\Lambda^\vee$, i.e. $G$ is adjoint. Otherwise, by Theorem \ref{teo: smooth-timashev} i) it must be $\Supp(\lambda) = \{\alpha_1, \ldots, \alpha_{r-1}\}$, while by
Propositions \ref{prop: fattorialita} and \ref{prop: Q-fattorialita} there exists an element $v_1 \in \mathcal{B}(\lambda)$ which is a multiple of $\omega_1^\vee$.
Notice that $\alpha_1^\vee \not \in \Supp(v)$ for every $v \in \mathcal{B}(\lambda) \smallsetminus \{v_1\}$: since $\langle \alpha_1, \alpha_2^\vee \rangle = -1$ and since
$\alpha_2^\vee \in \mathcal{X}(T)^\vee$, it follows then $v_1 = -\omega_1^\vee$. Hence $\mathcal{X}(T)^\vee$ is the coweight lattice and $G$ is adjoint.

\textit{Type $\mathsf{C}$.} If $\Delta$ has type $\mathsf{C}_r$ with $r \geqslant 1$, then $\Delta^\vee$ has type $\mathsf{B}_r$ and $\Lambda^\vee/\mathbb{Z}[\Delta^\vee]$ is generated by
$\omega_r^\vee$. Therefore Lemma \ref{lem: copesi fondamentali} implies that, if $\alpha_r \not \in \Supp(\lambda)$, then $\mathcal{X}(T)^\vee$ is the coweight lattice
and $G$ is adjoint. Hence we get $\Supp(\lambda) = \{\alpha_k, \ldots, \alpha_r\}$.

\textit{Type $\mathsf{D}$.} If $\Delta$ is of type $\mathsf{D}_r$, then $\Lambda^\vee/\mathbb{Z}[\Delta^\vee]$ is generated by any two fundamental coweights associated to some
simple root in $\Delta^e$. Therefore if $|\Supp(\lambda) \cap \Delta^e| \leqslant 1$ Lemma \ref{lem: copesi fondamentali} implies that $G$ is adjoint. Suppose that
$|\Supp(\lambda) \cap \Delta^e| = 2$. If $\alpha_1 \not \in \Supp(\lambda)$, then by Proposition \ref{prop: Q-fattorialita} we get that
$\Supp(\lambda)=\{\alpha_i,\ldots,\alpha_r\}$ for some $i\leq r-2$ and by Propositions \ref{prop: fattorialita} and \ref{prop: Q-fattorialita} they exist elements
$v_{r-1}, v_r \in \mathcal{B}(\lambda)$ which are multiples resp. of $\omega_{r-1}^\vee$ and of $\omega_r^\vee$. Notice that $\Supp(v) \cap \{\alpha^\vee_{r-1},
\alpha_r^\vee\} = \varnothing$ for every $v \in \mathcal{B}(\lambda) \smallsetminus\{v_{r-1},v_r\}$: since $\alpha_{r-2}^\vee \in \mathcal{X}(T)^\vee$ and since $\langle \alpha_{r-1},
\alpha_{r-2}^\vee \rangle = \langle \alpha_r, \alpha_{r-2}^\vee \rangle = -1$, it follows then $v_{r-1} = -\omega_{r-1}^\vee$ and $v_r = -\omega_r^\vee$. Hence
$\mathcal{X}(T)^\vee$ is the coweight lattice and $G$ is adjoint.

Suppose now that $\alpha_1 \in \Supp(\lambda)$: up to an automorphism of $\Phi$, by Proposition \ref{prop: Q-fattorialita} we may assume that
$\Supp(\lambda)=\{\alpha_1,\ldots,\alpha_{r-1}\}$. By Propositions \ref{prop: fattorialita} and \ref{prop: Q-fattorialita} there exist elements $v_1, v_{r-1} \in
\mathcal{B}(\lambda)$ which are multiples resp. of $\omega_1^\vee$ and of $\omega_{r-1}^\vee$. Notice that $\Supp(v) \cap \{\alpha_1^\vee, \alpha^\vee_{r-1}\} = \varnothing$ for
every $v \in \mathcal{B}(\lambda) \smallsetminus \{v_1, v_{r-1}\}$: since $\{\alpha_2^\vee, \alpha_{r-2}^\vee \} \subset \mathcal{X}(T)^\vee$ and since $\langle \alpha_1, \alpha_2^\vee
\rangle = \langle \alpha_{r-2}, \alpha_{r-1}^\vee \rangle = -1$, it follows then $v_1 = -\omega_1^\vee$ and $v_{r-1} = -\omega_{r-1}^\vee$. Hence $\mathcal{X}(T)^\vee$
is the coweight lattice and $G$ is adjoint.

\textit{Type $\mathsf{E}_6$.} If $\Delta$ is of type $\mathsf{E}_6$, then $\Lambda^\vee$ is generated both by $\omega_1^\vee$ and by $\omega_6^\vee$: hence by Lemma \ref{lem:
copesi fondamentali} it follows that $G$ is adjoint if $\Supp(\lambda) \cap \{\alpha_1, \alpha_6\} \neq \varnothing$. If this is not the case, then by Theorem
\ref{teo: smooth-timashev} i) it follows that $\Supp(\lambda)=\{\alpha_1,\alpha_3,\alpha_4,\alpha_5,\alpha_6\}$. By Propositions \ref{prop: fattorialita} and \ref{prop:
Q-fattorialita} there exists an element $v_1 \in \mathcal{B}(\lambda)$ which is a multiple of $\omega_1^\vee$. Notice that $\alpha_1^\vee \not \in \Supp(v)$ for every
$v \in \mathcal{B}(\lambda) \smallsetminus \{v_1\}$: since $\langle \alpha_1, \alpha_3^\vee \rangle = -1$ and since $\alpha_3^\vee \in \mathcal{X}(T)^\vee$, it follows then $v_1 =
-\omega_1^\vee$. Hence $\mathcal{X}(T)^\vee$ is the coweight lattice and $G$ is adjoint.

\textit{Type $\mathsf{E}_7$.} If $\Delta$ has type $\mathsf{E}_7$, then $\Lambda^\vee$ is generated both by $-\omega_2^\vee$ and by $-\omega_7^\vee$: hence by Lemma \ref{lem:
copesi fondamentali} it follows that $G$ is adjoint if $\Supp(\lambda) \cap \{\alpha_2, \alpha_7\} \neq \varnothing$. If this is not the case, then by Theorem
\ref{teo: smooth-timashev} i) it follows that $\Supp(\lambda) \supset \{\alpha_2,\alpha_4,\alpha_5,\alpha_6,\alpha_7\}$. By Propositions \ref{prop: fattorialita} and
\ref{prop: Q-fattorialita} there exists an element $v_2 \in \mathcal{B}(\lambda)$ which is a multiple of $\omega_2^\vee$. Notice that $\alpha_2^\vee \not \in \Supp(v)$
for every $v \in \mathcal{B}(\lambda) \smallsetminus \{v_2\}$: since $\langle \alpha_2, \alpha_4^\vee \rangle = -1$ and since $\alpha_4^\vee \in \mathcal{X}(T)^\vee$, it follows then
$v_2 = -\omega_2^\vee$. Hence $\mathcal{X}(T)^\vee$ is the coweight lattice and $G$ is adjoint.
\end{proof}

\end{document}